\documentclass[10pt]{article}

\usepackage{amsmath,amssymb,amsthm,mathrsfs,dsfont,graphicx}

\usepackage[margin=3cm]{geometry} %¿í°æÊ±ÓÃ£¬ÅäºÏË«±¶ÐÐ¾à

\usepackage{titlesec,hyperref}
\usepackage{esint}
\usepackage{color}

\usepackage{fancyhdr}
\titleformat{\subsection}{\it}{\thesubsection.\enspace}{1pt}{}

\newtheorem{theo}{Theorem}[section]

\newtheorem{rema}[theo]{Remark}

\newtheorem{lemm}[theo]{Lemma}
\newtheorem{defi}[theo]{Definition}

\newtheorem{prop}[theo]{Proposition}

\numberwithin{equation}{section}

\allowdisplaybreaks %ÔÊÐí¹«Ê½·ÖÒ³ÏÔÊ¾

\newcommand{\ii}{\textup{i}}

\usepackage[backend=biber,style=numeric]{biblatex} 
\begin{document}

\title{Non-uniqueness for the hyperdissipative Navier-Stokes equations with arbitrarily small subcritical data
\hspace{-4mm}
}

\author{Zipeng $\mbox{Chen}^1$ \footnote{Email: chenzp26@mail2.sysu.edu.cn},\quad Song $\mbox{Liu}^1$ \footnote{Email: lius37@mail2.sysu.edu.cn}\quad
	 and\quad
	Zhaoyang $\mbox{Yin}^{1}$\footnote{E-mail: mcsyzy@mail.sysu.edu.cn}\\
    $^1\mbox{School}$ of Science,\\ Shenzhen Campus of Sun Yat-sen University, Shenzhen 518107, China}
 
\date{}
\maketitle
\hrule

\begin{abstract}
 In this paper, we consider the hyperdissipative Navier-Stokes equations with fractional dissipation $(-\Delta)^{\beta}$ with $\beta>1$. We prove that smooth solutions of the hyperdissipative Navier-Stokes equations are non-unique with arbitrarily small initial data in ${B}^{-\beta-\alpha}_{\infty,1}(\mathbb{T}^d)$ for any $\alpha>0$. Moreover, we show the existence of a solution with arbitrarily small initial data in ${B}^{-\beta-\alpha}_{\infty,1}(\mathbb{T}^d)$ ($\alpha>0$) that grows arbitrarily large in
 $\dot{B}^{-s}_{\infty,\infty}(\mathbb{T}^d)$ for all $s\in\mathbb{R}$ in arbitrarily small time. It is worth pointing out that ${B}^{-\beta-\alpha}_{\infty,1}(\mathbb{T}^d)$ lies in the subcritical regime when $0<\alpha<\beta-1$. To the best of our knowledge, this is the first non-uniqueness result of the Navier-Stokes equations with initial data at the subcritical regularity. To show the sharpness of the above results, we establish the local well-posedness of the hyperdissipative Navier-Stokes equations with initial data in $\dot{B}^{-\beta-\alpha}_{\infty,\infty}(\mathbb{T}^d)$ with $\alpha< 0$.
\end{abstract}
\noindent {\sl Keywords:} Non-uniqueness, Navier-Stokes equations, 

\vskip 0.2cm

\noindent {\sl AMS Subject Classification:} 35Q30, 76D03  \

\vspace*{10pt}

\tableofcontents

\section{Introduction }
  In this paper, we consider the following $d$-dimensional Navier-Stokes equations with fractional dissipation on $\mathbb{T}^d$:
\begin{equation}\label{e:NSf}
\begin{cases}
\partial_tv+ \operatorname{div}(v\otimes v)+(-\Delta)^{\beta} v+\nabla p=0, \\
\operatorname{div}\,v=0,
\end{cases}
\end{equation}
where $d\geq2$, $\beta>\frac{1}{2}$, $\mathbb{T}^d=\mathbb{R}^d/\mathbb{Z}^d$ is the $d$-dimensional torus and the fractional dissipation operator $(-\Delta)^\beta$ on the flat torus is defined via the Fourier transform by
\begin{gather*}
    \mathcal{F}((-\Delta)^\beta v)(\xi)=|\xi|^{2\beta}\mathcal{F}(v)(\xi),\quad\forall\xi\in\mathbb{T}^d
\end{gather*}
Here, $v=(v_1,\dots,v_d)^{\top}\in\mathbb{R}^d$ and $p=p(t,x)\in\mathbb{R}$ represent the flow velocity and the scalar pressure, respectively. The generalized Navier-Stokes equation \eqref{e:NSf} are invariant under the scaling
\begin{gather}
    v(t,x)\mapsto\lambda^{2\beta-1} v(\lambda^{2\beta}t,\lambda x ),\quad p(t,x)\mapsto\lambda^{4\beta-2}p(\lambda^{2\beta}t,\lambda x).\label{scale}
\end{gather}
Accordingly, a function space is called critical for \eqref{e:NSf} if its norm is invariant under the scaling transformation \eqref{scale}.

When $\beta = 1$, system (\ref{e:NSf}) takes the form of the standard incompressible Navier–Stokes equations:
\begin{equation}\label{e:NS}
\begin{cases}
\partial_tv+ \text{div}(v\otimes v)-\Delta v+\nabla p=0, \\
\text{div}\,v=0,
\end{cases}
\end{equation}
In the remarkable paper \cite{leray}, for any dimension $d\geq2$, Leray first showed that there exists a weak solution in $L^\infty(\mathbb{R}^+;L^2(\mathbb{R}^d))\cap L^2(\mathbb{R}^+;\dot{H}^1(\mathbb{R}^d))$ for any $L^2$ solenoidal initial data, which satisfies the energy inequality
\begin{gather*}
    \|v(t)\|^2_{L^2}+2\int^t_0\|\nabla v(s)\|^2_{L^2}ds\leq\|v(0)\|^2_{L^2},\quad\forall t\geq0.
\end{gather*}
For smooth bounded domains with Dirichlet boundary conditions, Hopf \cite{hopf} established an analogous result. This class of weak solutions is now known as Leray–Hopf weak solutions. 
To this day, the uniqueness problem for Leray–Hopf weak solutions of the Navier–Stokes equations remains an open question in dimensions greater than two. For $\beta\geq\frac{5}{4}$, the Leray-Hopf weak solution class is scaling-critical or scaling-subcritical. In this regime, Lions \cite{lion54} proved the existence and uniqueness of Leray-Hopf weak solutions to the initial-value problem.

 Convex integration techniques for the Euler equations have been extensively developed; see, for instance, \cite{introductionconvex1,introductionconvex2,Isett3,B1,2donsager}.
In the breakthrough work \cite{ns有限能量不唯一}, Buckmaster-Vicol first showed non-uniqueness of finite energy weak solutions to the 3D Navier-Stokes equations using the convex integration framework. The key ingredient in \cite{ns有限能量不唯一} is $L^2_x$-based intermittent spatial building block, which enables control of the dissipativity term $(-\Delta)u$ in the Navier-Stokes equations.
Following the convex integration constructions developed in \cite{ns有限能量不唯一}, the authors of \cite{buckmaster1} and \cite{ns有限能量不唯一低于lion指标} extend the non-uniqueness results to the hyperdissipative Navier–Stokes equations with dissipation exponent $\beta<\frac{5}{4}$, below the Lions’ exponent. For applications to other models, we refer the reader to \cite{GeolemmaMHD,modena,transport1,transport2,stationary1}.

The weak-strong uniqueness property for (\ref{e:NS}) in critical spaces has also attracted considerable attention in the literature. The Lady\v{z}enskaja-Prodi-Serrin criteria \cite{kozono,lady,Prodi,serrin} states that if a Leray-Hopf weak solution belongs to the critical or subcritical regime:
\begin{gather*}
     L^p_tL^q_x\quad\mathrm{with}\quad\frac{2}{p}+\frac{d}{q}\leq1.
\end{gather*}
then uniqueness is guaranteed for Leray–Hopf solutions with the same initial data (more precisely, we need $C_tL^\infty_x $ rather than $L_tL^\infty_x $ when $p=\infty$). For the generalized Navier-Stokes equations, Zhou \cite{GLPS} established uniqueness of Leray-Hopf weak solution belonging to the critical or subcritical Serrin class
\begin{gather*}
     L^p_tL^q_x\quad\mathrm{with}\quad\frac{2\beta}{p}+\frac{d}{q}\leq2\beta-1.
\end{gather*}

In \cite{serrin准则luo}, Cheskidov-Luo employed the intermittent temporal building block to prove the non-uniqueness of (\ref{e:NS}) on $L^p_tL^\infty_x$ for all $p<2$. Moreover, Cheskidov-Luo \cite{MR4610908} established the non-uniqueness on the class $L^{\infty}_tL^q_x$ for all $q<2$ for the 2D Navier-Stokes equations. Both results are sharp with respect to the Lady\v{z}enskaja-Prodi-Serrin criteria.
For the 3D Navier-Stokes equations with hyperviscosity exponent $\beta$ beyond the Lion exponent $\frac{5}{4}$, Li-Qu-Zeng-Zhang \cite{qupeng} also showed sharp non-uniqueness in some supercritical spaces. For applications to other models, we refer the reader to \cite{MHDzeng1,yeMHD,Boussinesqchen}.

By the mild formulation of equations, the global well-posedness for small data and local well-posedness for large data for the standard incompressible Navier–Stokes equations have been established in certain critical spaces
\begin{gather*}
    \dot{H}^{\frac{d}{2}-1}\subset L^d\subset \dot{B}^{-1+\frac{d}{p}}_{p,\infty}\subset BMO^{-1}\subset \dot{B}^{-1}_{\infty,\infty}
\end{gather*}
where $d< p<\infty$. 
Fujita-Kato\cite{30} and Kato\cite{40} proved that this holds for initial data in the first two spaces, respectively. A corresponding result in Besov spaces $\dot{B}^{-1+\frac{d}{p}}_{p,\infty}(\mathbb{R}^d) (d< q<\infty)$ was obtained by Cannone \cite{13} and Planchon \cite{53}. Koch-Tataru\cite{BMO-1wp} established that global well-posedness holds for small data in $BMO^{-1}$. For the fractional Navier-Stokes equations with $\beta\in (\frac{1}{2},1)$, Zhai \cite{BMO-2beta-1} and Zhai-Yu \cite{B-2beta-1} established global well-posedness for small initial data in $BMO^{-(2\beta-1)}$ and even in the largest scaling-invariant space $\dot{B}^{-(2\beta-1)}_{\infty,\infty}$.

Nevertheless, ill-posedness will also arise in the critical class or even subcritical space. For the largest critical space $\dot{B}^{-1}_{\infty,\infty}$, Bourgain-Pavlovi\'c \cite{Bourgain} showed that norm inflation is possible in 3D Navier-Stokes equations. In other words, the solutions with arbitrarily small data in $\dot{B}^{-1}_{\infty,\infty}$ can grow arbitrarily large in a short time. On the other hand, Germain \cite{33} proved that the data-to-solution map is not $C^2$ in the spaces $B^{-1}_{\infty,q}$ for $q>2$ and Yoneda \cite{63} showed that ill-posedness holds in this class. Later in \cite{60}, Wang showed that norm inflation from small data can occur even in the spaces $B^{-1}_{\infty,q}$ for $q\in [1,2]$, which are continuously embedded in $BMO^{-1}$. Palasek \cite{normgrowthBMO} constructed a family of smooth initial data bounded in $BMO^{-1}$ whose corresponding solutions exhibit arbitrarily large norm growth.
Cheskidov-Dai \cite{dai14} established the phenomenon of norm inflation with small data in subcritical space $\dot{B}^{-\beta}_{\infty,\infty}$ for the generalized Navier–Stokes equations with dissipation exponent $\beta>1$.

For the non-uniqueness results, Coiculescu-Palasek \cite{Coiculescu2025} first constructed two distinct global solutions of 3D Navier-Stokes equations with identical initial data in critical space $BMO^{-1}$. Subsequently, Miao-Nie-Ye \cite{miao2026nonuniquenesssmoothsolutionsnavierstokes} established an analogous result in the two-dimensional case. See also \cite{daimimi1,daimimi2,5dMHD,2dMHD}.

We also refer to another programme by Jia and \v{S}ver\'ak \cite{jia2,jia1}. They demonstrated that non-uniqueness of Leray–Hopf weak solutions is valid, provided that a suitable spectral condition holds for the linearized Navier–Stokes operator. Building on the instability framework of \cite{vishik1,vishik2}, Albritton-Bru\'e-Colombo \cite{forcedns} proved non-uniqueness in the Leray-Hopf weak solutions for the forced 3D Navier-Stokes equations. See also the related work \cite{abchypodissipativens} for the analogous result in the setting of the two-dimensional hypodissipative Navier--Stokes equations.

\subsection{Main result}
 Our first result concerns the non-uniqueness of solutions to the hyperdissipative Navier–Stokes equations arising from subcritical initial data.

\begin{theo}[Non-uniqueness]\label{maintheo1}
    Let $d\geq2$, $\beta>1$ and $\alpha>0$. For any $\varepsilon>0$ and $0<p<\frac{2\beta}{\beta+\alpha}$, there exist divergence-free initial data $V^0\in B^{-\beta-\alpha}_{\infty,1}(\mathbb{T}^d)$ such that 
    $\|V^0\|_{B^{-\beta-\alpha}_{\infty,1}(\mathbb{T}^d)}<\varepsilon$ and
    the hyperdissipative Navier-Stokes equation (\ref{e:NSf}) admits two distinct solutions
    \begin{align*}
        v^{(1)},v^{(2)}\in &C^\infty_{t,x}((0,1]\times\mathbb{T}^d)\cap C^0([0,1];B^{-\beta-\alpha}_{\infty,1}(\mathbb{T}^d))\\
        & \cap L^{p}([0,1];L^\infty(\mathbb{T}^d)).
    \end{align*}
\end{theo}

\begin{rema}
    When $\alpha<\beta-1$, $B^{-\beta-\alpha}_{\infty,1}(\mathbb{T}^d)$ lies in the subcritical regime. To our best knowledge, this provides the first non-uniqueness result for the Navier–Stokes equations in subcritical spaces.
    When $\alpha<\beta-1$ and $p\in[\frac{2\beta}{2\beta-1},\frac{2\beta}{\beta+\alpha})$, the space $L^{p}([0,1];L^\infty(\mathbb{T}^d))$ is subcritical and lies below the classical Ladyzhenskaya–Prodi–Serrin threshold. Nevertheless, the initial data belong to the rougher space $B^{-\beta-\alpha}_{\infty,1}(\mathbb{T}^d)$ rather than the energy space $L^2(\mathbb{T}^d)$. Therefore, establishing a non-uniqueness result in this setting does not contradict the weak–strong uniqueness theorem.
\end{rema}

The second result establishes the phenomenon of norm inflation in the context of the hyperdissipative Navier–Stokes equations.
\begin{theo}[Norm inflation]\label{maintheo2}
    Let $d\geq2$, $\beta>1$, $\alpha>0$ and $s\in\mathbb{R}$. For any $\delta>0$, there exists a solution $v$ to the hyperdissipative Navier-Stokes equation (\ref{e:NSf}) and $0<t<\delta$, such that
    \begin{gather*}
        \|v(0)\|_{B^{-\beta-\alpha}_{\infty,1}(\mathbb{T}^d)}<\delta,
    \end{gather*}
    with
    \begin{gather*}
        \|v(t)\|_{B^{-s}_{\infty,\infty}(\mathbb{T}^d)}>\delta^{-1}.
    \end{gather*}
\end{theo}

To demonstrate the sharpness of Theorem \ref{maintheo1} and Theorem \ref{maintheo2}, we now establish the local well-posedness of the generalized Navier–Stokes equations.
\begin{theo}[Local well-posedness]\label{maintheo3}
    Let $ \beta > \frac{1}{2}$ and $\alpha<\min\{0,\beta-1\}$. For any divergence-free initial data $ v_0 \in \dot{B}_{\infty,\infty}^{-\beta-\alpha}(\mathbb{T}^d)$, there exists $T(v_0)>0$ such that the generalized Navier-Stokes equation (\ref{e:NSf}) admits a unique solution $ v \in L^\infty\big((0,T(v_0)], \dot{B}_{\infty,\infty}^{-\beta-\alpha}(\mathbb{T}^d)\big)$ satisfying
\[
\sup_{t\in (0,T(v_0)]} t^{\frac{\beta+\alpha}{2\beta}} \|v(t)\|_{L^\infty(\mathbb{T}^d)} < \infty.
\]
\end{theo}

\begin{theo}[Well-posedness with small data]\label{maintheo4}
    Let $ \beta > \frac{1}{2}$ and $\alpha\leq\beta-1$ with $\alpha<0$, and let $T>0$. There exists $\varepsilon_\beta > 0$ such that for any divergence-free initial data $ v_0 \in \dot{B}_{\infty,\infty}^{-\beta-\alpha}(\mathbb{T}^d)$ satisfying
\[
\|v_0\|_{\dot{B}_{\infty,\infty}^{-\beta-\alpha}(\mathbb{T}^d)} \le \varepsilon_\beta,
\]
the generalized Navier-Stokes equation (\ref{e:NSf}) admits a unique solution $ v \in L^\infty\big((0,T], \dot{B}_{\infty,\infty}^{-\beta-\alpha}(\mathbb{T}^d)\big)$ such that
\[
\sup_{t\in (0,T]} t^{\frac{\beta+\alpha}{2\beta}} \|v(t)\|_{L^\infty(\mathbb{T}^d)} < \infty.
\]
Moreover, if $\beta\in(\frac{1}{2},1)$ and $\alpha=\beta-1$, the existence time $T$ can be taken to be infinite.
\end{theo}

\begin{rema}
     Combing the thresholds in Theorem \ref{maintheo3} and Theorem \ref{maintheo4} with the fact that $\alpha$ is required to be positive in Theorem \ref{maintheo1}, naturally leads to a series of interesting questions:
    \begin{itemize}
        \item Is it possible to construct a non-uniqueness example with arbitrary small initial data in $\dot{B}_{\infty,\infty}^{-\beta}(\mathbb{T}^d)$ for $\beta\geq1$?
        \item Is it possible to construct a non-uniqueness example with large initial data in $\dot{B}_{\infty,\infty}^{-(2\beta-1)}(\mathbb{T}^d)$ or even in $BMO^{-(2\beta-1)}\triangleq(-\Delta)^{-\frac{2\beta-1}{2}}BMO$ for $\beta\in(\frac{1}{2},1)$?
    \end{itemize}
   
\end{rema}

Since the solutions in the lower dimensions satisfy the higher-dimensional generalized Navier-Stokes equations when zero components are added to the remaining dimensions. Moreover, the embedding $B^{a}_{\infty,\infty}\hookrightarrow B^{b}_{\infty,1}$ holds for all $a>b$. Consequently, proving Theorem \ref{maintheo1} and Theorem \ref{maintheo2} reduces to establishing the following propositions.

\begin{prop}\label{mainprop1}
    Let $d=2$, $\beta>1$ and $0<\alpha< \min\{\frac{1}{2},\beta-1\}$. For any $\varepsilon,\varepsilon'>0$ and $0<p<\frac{2\beta}{\beta+\alpha}$, there exist divergence-free initial data $V^0\in B^{-\beta-\alpha}_{\infty,\infty}(\mathbb{T}^2)$ such that 
    $\|V^0\|_{B^{-\beta-\alpha-\varepsilon}_{\infty,\infty}(\mathbb{T}^2)}<\varepsilon'$ and
    the hyperdissipative Navier-Stokes equation (\ref{e:NSf}) admits two distinct solutions
    \begin{align*}
        v^{(1)},v^{(2)}\in &C^\infty_{t,x}((0,1]\times\mathbb{T}^2)\cap C^0([0,1];B^{-\beta-\alpha-\varepsilon}_{\infty,\infty}(\mathbb{T}^2))\\
        & \cap L^{p}([0,1];L^\infty(\mathbb{T}^2)).
    \end{align*}
\end{prop}

\begin{prop}\label{mainprop2}
    Let $d=2$, $\beta>1$, $0<\alpha< \min\{\frac{1}{2},\beta-1\}$ and $s>\beta+\alpha$. For any $\delta>0$, there exists a solution $v$ to the hyperdissipative Navier-Stokes equation (\ref{e:NSf}) and $0<t<\delta$, such that
    \begin{gather*}
        \|v(0)\|_{B^{-\beta-\alpha-\varepsilon}_{\infty,\infty}(\mathbb{T}^2)}<\delta,
    \end{gather*}
    with
    \begin{gather*}
        \|v(t)\|_{B^{-s}_{\infty,\infty}(\mathbb{T}^2)}>\delta^{-1}.
    \end{gather*}
\end{prop}

\subsection{Main idea of the construction}
\subsubsection{Non-uniqueness}
The non-uniqueness mechanism of the hyperdissipative Navier-Stokes equations follows from the approach of \cite{Palasek2025} and \cite{Coiculescu2025}. In this part, we outline the main idea of the proof, omitting certain details for the sake of exposition. In our construction, there exist two types of flows:

\noindent
$\bullet$ The first is the heat-dominated flow, which is the solution of heat equation with fractional dissipation, up to a small error. 
\begin{align*}
        \partial_tv_k+(-\Delta)^{\beta} v_k=l.o.t.\\
\end{align*}
where $l.o.t$ means some lower order terms. 
And the heat-dominated flows satisfy the following exponential decay in time:
\begin{gather*}
    \|v_k\|_{L^\infty(\mathbb{T}^2)}\lesssim N_k^{\beta+\alpha} e^{-N_{k}^{2\beta}t}.
\end{gather*}

\noindent
$\bullet$ The second is the inverse cascade-dominated flow, whose temporal partial derivative annihilates the nonlinear self-interactions of the heat-dominated flow.
\begin{align*}
        \partial_t\bar{v}_k+\mathbb{P}\operatorname{div} (v_{k+1}\otimes v_{k+1})=l.o.t.
\end{align*}
Compared with the heat-dominated flows, the inverse cascade-dominated flows exhibit a faster exponential temporal decay:
\begin{gather*}
    \|\bar{v}_k\|_{L^\infty(\mathbb{T}^2)}\lesssim N_k^{\beta+\alpha} e^{-N_{k+1}^{2\beta}t}.
\end{gather*}

The principal part of the solutions to the hyperdissipative Navier-Stokes equations is derived and given as follows:
\begin{gather*}
    v^{(1)}\triangleq\sum_{k\geq0\,even}v_k+\sum_{k\geq0\,odd}\bar{v}_k,
    \end{gather*}
    and 
    \begin{gather*}
    v^{(2)}\triangleq\sum_{k\geq0\,odd}v_k+\sum_{k\geq0\,even}\bar{v}_k.
\end{gather*}

In the following discussion, $i$ denotes either $1$ or $2$. As a result of the cancellation between the heat-dominated flows and the inverse cascade-dominated flows, we can deduce remaining terms $F^{(i)}$ and $\widetilde{F}^{(i)}$ which are defined by
\begin{align*} 
        \partial_t v^{(i)}+(-\Delta)^{\beta} v^{(i)}+ \mathbb{P}\operatorname{div} (v^{(i)}\otimes  v^{(i)})=-\mathbb{P}\operatorname{div}F^{(i)}-\Delta^{[\beta]}\widetilde{F}^{(i)},  
\end{align*}
are arbitrarily small in some subcritical norms. In other words, we conclude that $v^{(i)}$ is a solution to the hyperdissipative Navier-Stokes equation \eqref{e:NSf} with a small error in divergence form and another error in Laplacian form. 

In the final step, we add the perturbation $\omega^{(i)}$ to $v^{(i)}$ to correct the associated errors, so that $v^{(i)}+\omega^{(i)}$ is indeed a solution to the hyperdissipative Navier-Stokes equations \eqref{e:NSf}. Then $\omega^{(i)}$ must satisfy 
\begin{align*}
\begin{cases}
            \partial_t \omega^{(i)}+(-\Delta)^{\beta} \omega^{(i)}+\mathbb{P}\operatorname{div}\, (v^{(i)}\otimes \omega^{(i)}+\omega^{(i)}\otimes v^{(i)}+\omega^{(i)}\otimes \omega^{(i)})=\mathbb{P}\operatorname{div}F^{(i)}+\Delta^{[\beta]}\widetilde{F}^{(i)},\\
            \omega^{(i)}(0,\cdot)=0.
        \end{cases}
\end{align*}
Given that $F^{(i)}$ and $\widetilde{F}^{(i)}$ are small in some certain subcritical spaces, we are able to establish the existence of $\omega^{(i)}$ and confirm its smallness in appropriate subcritical norms by means of a fixed point argument. 

Using the fact that $\bar{v}_k(0,\cdot)=v_k(0,\cdot)$ and $\omega^{(i)}(0,\cdot)=0$, we know that ${v}^{(1)}+\omega^{(1)}$ and ${v}^{(2)}+\omega^{(2)}$ have the same initial data in $B^{-\beta-\alpha}_{\infty,\infty}(\mathbb{T}^2)$. However, owing to the distinct temporal decay behaviors of $v_k$ and $\bar{v}_k$, we can infer that ${v}^{(1)}+\omega^{(1)}$ and ${v}^{(2)}+\omega^{(2)}$ are distinct (see details in Section \ref{Proof of the Non-Uniqueness}).

In the following, we provide some remarks on the construction.

\textbf{(i)} 
For the case $\beta=1$, the solutions constructed in \cite{Coiculescu2025} satisfy all the conclusions of Theorem \ref{maintheo1} when $\beta=1$.
In the hyperdissipative regime $\beta>1$, Theorem \ref{maintheo1} and Theorem \ref{maintheo2} establish non-uniqueness and norm inflation results for arbitrarily small initial data in subcritical spaces. This stems from the fact that the exponent required in the method of \cite{Coiculescu2025} is strictly below the critical threshold, namely $\beta<2\beta-1$. Thus, in the arguments of the functions of the Geometric Lemma, we can absorb largeness arising from size of the $(k-1)$-th velocity potential by using the additional factor $\frac{1}{N_k^{2\alpha}}$ (see Definition \ref{Apropdefivk}). 

\textbf{(ii)} For the structure of the principal part, \cite{Coiculescu2025} introduced Mikado flows. By the intersection of the supports of the Mikado flows, the supports of $v_k(0,\cdot)$ will shrink as $k$ increases. This localization property ensures that the initial data lies in $BMO^{-1}$. However, in the present paper, we are not concerned with proving that the initial data belong to spaces such as $BMO^{-(2\beta-1)}$. Rather, we work in a class of Besov spaces. Consequently, it is not necessary to introduce Mikado flows in our main part, which simplifies the construction. 

 \textbf{(iii)}  Suppose that the remaining terms are only written in divergence form as \cite{Coiculescu2025}. In order to derive the perturbation, we need to ensure that
 \begin{gather*}
     \int_0^te^{(t-s)(-\Delta)^\beta}\mathbb{P}\operatorname{div}F^{(i)}(s)ds
 \end{gather*}
 is well-defined. So it suffices to control the residual errors in divergence form, $\mathbb{P}\operatorname{div}F^{(i)}$, in the norm 
 \begin{align*}
     \|F^{(i)}\|_Y\triangleq\sup_{t\in(0,1]}t^{1-\tau}\|F^{(i)}(t)\|_{L^{\infty}}
 \end{align*}
 where $\tau>0$ is a small parameter. However, such an estimate is available only when $1<\beta<2$. More precisely, since $v^{(i)}$ belongs to critical or subcritical spaces, depending on the choice of $\alpha$, the dominant contribution to $F^{(i)}$ arises from the linear error term. By applying the fractional Leibniz commutator estimate, the linear error term can be controlled only for $1<\beta<2$. Nevertheless, when $\beta\geq2$, the corresponding estimate is no longer sufficiently subcritical, and consequently $F^{(i)}\notin Y$.
 
 To overcome this difficulty, we introduce a new structure $\Delta^{[\beta]}\widetilde{F}^{(i)}$. Specifically, we recast the linear error term as $\Delta^{[\beta]}\widetilde{F}^{(i)}$, whereas all other error terms are retained in the form $\mathbb{P}\operatorname{div}F^{(i)}$. Here $F^{(i)}$ belong to $Y$.
 We verify that $\nabla^{\kappa}\widetilde{F}^{(i)}$ are small in $Y$, where $\kappa>0$ is a small parameter. Therefore, we deduce 
 \begin{gather*}
     \|\int_0^te^{(t-s)(-\Delta)^\beta}\Delta^{[\beta]}\widetilde{F}^{(i)}(s)ds\|_{L^{\infty}}\lesssim \|\nabla^{\kappa}\widetilde{F}^{(i)}\|_Y\int_0^t (t-s)^{-\frac{2[\beta]-\kappa}{2\beta}}s^{-1+\tau}ds\lesssim \|\nabla^{\kappa}\widetilde{F}^{(i)}\|_Y,
 \end{gather*}
 which can establishes the existence of the perturbation $\omega^{(i)}$.

\textbf{(iv)} In the two-dimensional case, one encounters technical difficulties in controlling the oscillatory error, similar to those faced in the approach to 2d Onsager's conjecture. Specifically, these difficulties stem from the overlap of the modes
\begin{gather*}
    \{e^{2\pi i N_k\xi\cdot x}\bar{\xi}\}_{\xi\in\Lambda}
\end{gather*}
with different orientations in the nonlinear term of the Navier-Stokes equations. To overcome this difficulty, we adopt the approach of \cite{miao2026nonuniquenesssmoothsolutionsnavierstokes}, which allows the excess oscillatory error to be absorbed into the pressure term.

\subsubsection{Norm inflation}
The norm inflation mechanism for the hyperdissipative Navier–Stokes equations is inspired by \cite{bmonormgrowth}. In this part, we sketch the main idea of the proof, omitting certain technical details for clarity. 

The principal perturbation is constructed as the sum of two components with distinct frequency scales,
\begin{align*}
    U=U_1+U_2,
\end{align*}
where $U_1$ satisfies the vanishing self-interaction condition and $U_2$ is the heat-dominated flow
\begin{align*}
    \operatorname{div}(U_1\otimes U_1)=0\quad\mathrm{and}\quad (\partial_t+(-\Delta)^{\beta})U_2=l.o.t.
\end{align*}
Furthermore, the quadratic self-interaction of $U_2$ is designed to balance the  evolution of $U_1$ up to lower-order terms:
\begin{align*}
    (\partial_t+(-\Delta)^{\beta})U_1+\mathbb{P}\operatorname{div}(U_2\otimes U_2)=l.o.t.
\end{align*}
Hence, by the properties of $U_1$ and $U_2$, the residual errors $G$ and $\widetilde{G}$, defined through
\begin{align*} 
        \partial_t U+(-\Delta)^{\beta} U+ \mathbb{P}\operatorname{div} (U\otimes  U)=-\mathbb{P}\operatorname{div}G-\Delta^{[\beta]}\widetilde{G},  
\end{align*}
can be made arbitrarily small in suitable subcritical norms.  Following the strategy of the non-uniqueness proof, we construct a perturbation $\rho$ with zero initial data that eliminates the remaining errors, so that $U+\rho$ is a solution to the hyperdissipative Navier--Stokes equations \eqref{e:NSf}.

For the norm inflation argument, $U_1$ and $U_2$ are constructed such that
\begin{align*}
    U(0,\cdot)=U_2(0,\cdot),
\end{align*}
thus the smallness of $U(0,\cdot)$ in $B^{-\beta-\alpha}_{\infty,\infty}$ follows directly from the corresponding estimate for $U_2(0,\cdot)$. Since $U_2$ is supported at substantially higher frequencies than $U_1$, $U_2$ decays faster than $U_1$. At a suitable time $t_q$, we have
\begin{align*}
    U(t_q,\cdot)+\rho(t_q,\cdot)=U_1(t_q,\cdot)+o(U_1(t_q,\cdot)) \quad\mathrm{in}\quad B^{-s}_{\infty,\infty}.
\end{align*}
Meanwhile, $\|U_1(t_q,\cdot)\|_{B^{-s}_{\infty,\infty}}$ is designed to be large, and consequently so is $\|U(t_q,\cdot)\|_{B^{-s}_{\infty,\infty}}$.

%Finally, we emphasize that our construction is considerably simpler than that in \cite{bmonormgrowth}. While the construction in \cite{bmonormgrowth} relies on a finite family of building blocks ${u_k}_{0\le k\le k^{*}}$ with $k^{*}<\infty$, where $u_0$ generates the norm inflation and $\sum_{1\le k \le k^{*}}u_k$ controls the initial data, our construction uses only two building blocks: $U_1$, which generates the norm inflation, and $U_2$, which controls the initial data.

 In \cite{bmonormgrowth}, Norm inflation is achieved by combining a family of building blocks $\{v_k\}_{0\le k\le k^{*}}$, where $k^{*}<\infty$ is chosen sufficiently large. The largeness of $k^{*}$ is designed to suppresses the influence of the initial data. By contrast, thanks to the structure of our construction and the function space we work in, the entire mechanism is reduced to only two building blocks, $U_1$ and $U_2$. The former generates the norm inflation, while the latter prescribes the initial data.

\section{Preliminaries}
\subsection{Besov space}
Let $\varphi,\chi\in C_c^\infty(\mathbb R^d)$ be nonnegative radial functions taking values in $[0,1]$ such that
\begin{align*}
\operatorname{supp}\varphi&\subset
\left\{\xi\in\mathbb R^d:
\frac34\le |\xi|\le\frac83
\right\},&\varphi&\equiv1\quad\text{on }
\left\{\xi\in\mathbb R^d:\frac67\le |\xi|\le\frac{12}{7}
\right\},
\\\operatorname{supp}\chi&\subset\left\{
\xi\in\mathbb R^d:|\xi|\le\frac43\right\},&\chi&\equiv1\quad\text{on }\left\{\xi\in\mathbb R^d:|\xi|\le\frac34\right\},
\end{align*}
and
\begin{align*}
\chi(\xi)+\sum_{j=1}^{\infty}\varphi(2^{-j}\xi)&=1,\qquad \xi\in\mathbb R^d,\\
\sum_{j\in\mathbb Z}\varphi(2^{-j}\xi)&=1,\qquad \xi\in\mathbb R^d\setminus\{0\}.
\end{align*}
Let $D(\mathbb{T}^d)$ denote the space of infinitely smooth functions on \(\mathbb{T}^d\) and $D'(\mathbb{T}^d)$ be defined as its topological dual. Every distribution $f\in D'(\mathbb{T}^d)$ admits the formal Fourier series
\begin{align*}
f(x)=\sum_{l\in\mathbb{Z}^d}\hat{f}(l)e^{2\pi \ii l\cdot x} ,
\end{align*}
where 
\begin{gather*}
    \hat{f}(l)=\fint_{\mathbb{T}^d} f(x)e^{-2\pi ix\cdot l}dx.
\end{gather*}

The nonhomogeneous dyadic blocks $\{\Delta_j\}_{j\geq-1}$ are defined by 
\begin{gather*}
    \Delta_{-1}f\triangleq\sum_{l\in\mathbb{Z}^d}\chi(l)\widehat{f}(l)e^{ 2\pi \ii l\cdot x}.
\end{gather*}
and, for each $j\geq0$,
\begin{align*}
\Delta_{j}f\triangleq\sum_{l\in\mathbb{Z}^d}\varphi(2^{-j}l)\widehat{f}(l)e^{ 2\pi \ii l\cdot x}.
\end{align*}
The homogeneous dyadic blocks $\{\dot{\Delta}_j\}_{j\in\mathbb{Z}}$ are defined for all $j\in\mathbb{Z}$ by 
\begin{align*}
\dot{\Delta}_jf\triangleq\sum_{l\in\mathbb{Z}^d}\varphi(2^{-j}l)\widehat{f}(l)e^{ 2\pi \ii l\cdot x}.
\end{align*}
Note that, on the torus, $\dot{\Delta}_jf=0$ whenever $j\leq-2$.

\begin{defi}\label{defibesov}
Let $s\in\mathbb{R}$ and $1\le p,q\le\infty$.  The nonhomogeneous Besov space $ B^s_{p,q}(\mathbb{T}^d)$ consists of all $\mathbb{T}^d$-periodic functions $u\in D'(\mathbb{T}^d)$  such that
\begin{align*}
\|u\|_{ B^s_{p,q}(\mathbb{T}^d)}&\triangleq\left(\sum_{j\geq-1}\left(2^{js}\|\Delta_j u\|_{L^p(\mathbb{T}^d)}\right)^{q}\right)^{\frac{1}{q}}<\infty, \quad\quad\text{if}\quad q<\infty,\\
\|u\|_{ B^s_{p,\infty}(\mathbb{T}^d)}&\triangleq\sup_{j\geq-1}2^{js}\|\Delta_j u\|_{L^p(\mathbb{T}^d)}<\infty, \quad\quad\quad\,\,\,\,\quad\quad\text{if}\quad q=\infty.
\end{align*}
And the homogeneous Besov space $ \dot{B}^s_{p,q}(\mathbb{T}^d)$ consists of all $\mathbb{T}^d$-periodic functions $u\in D'(\mathbb{T}^d)$  such that
\begin{align*}
\|u\|_{ \dot{B}^s_{p,q}(\mathbb{T}^d)}&\triangleq\left(\sum_{j\in\mathbb{Z}}\left(2^{js}\|\Delta_j u\|_{L^p(\mathbb{T}^d)}\right)^{q}\right)^{\frac{1}{q}}<\infty, \quad\quad\text{if}\quad q<\infty,\\
\|u\|_{ \dot{B}^s_{p,\infty}(\mathbb{T}^d)}&\triangleq\sup_{j\in\mathbb{Z}}2^{js}\|\Delta_j u\|_{L^p(\mathbb{T}^d)}<\infty, \quad\quad\quad\,\,\,\,\quad\quad\text{if}\quad q=\infty.
\end{align*}
\end{defi}

\subsection{Notation and parameters}
For $f\in L^2(\mathbb{T}^d)$, we define its Fourier coefficients by
\begin{gather*}
    \hat{f}(\xi)\triangleq\fint_{\mathbb{T}^d}f(x)e^{-2\pi i\xi\cdot x}ds,\quad\forall \xi\in\mathbb{Z}^d,
\end{gather*}
 and the inverse Fourier transform is given by
 \begin{gather*}
     \mathcal{F}^{-1}f(x)\triangleq\sum_{\xi\in\mathbb{Z}^d}f(\xi)e^{2\pi i\xi\cdot x},\quad\forall x\in\mathbb{Z}^d.
 \end{gather*}
 We denote by $\mathbb{P}$ the Leray projection onto divergence-free vector fields, namely
    \begin{align*}
        \mathbb{P}f \triangleq \mathcal{F}^{-1}((\mathrm{Id}-\frac{\xi\otimes\xi}{|\xi|^2})\hat{f}(\xi)),
    \end{align*}
for any vector-valued function $f$. We also define the fractional heat semigroup by
\begin{align*}
    e^{-t(-\Delta)^\beta}f\triangleq\mathcal{F}^{-1}(e^{-t|\xi|^{2\beta}}\hat{f}(\xi)).
\end{align*}
We define the convolution on $\mathbb{T}^2$ as follows. For any
$f,g\in C^{\infty}(\mathbb{T}^2;\mathbb{R}^2)$,
\begin{gather*}
    f*g(x)\triangleq\int_{\mathbb{T}^2}f(x-y)g(y)\,dy,\quad\forall x\in\mathbb{T}^2.
\end{gather*}

Let $S^{2\times2}(\mathbb{R})$ be the space of real symmetric $2\times2$ matrices. For $x\in\mathbb{R}$, we denote by $[x]$ the greatest integer less than or equal to $x$. For $a,b<1$, we denote the Beta integral
\begin{align*}
    I(a,b)\triangleq\int_0^1(1-s)^{-a}s^{-b}ds,
\end{align*}
whenever the integral is well-defined.

Throughout the rest of this paper, we focus on proving Proposition \ref{mainprop1} and Proposition \ref{mainprop2}. Thus, we fix $d=2$, $\beta>1$,  $0<\alpha<\min\{\frac{1}{2},\beta-1\}$, $s>\beta+\alpha$ and $\varepsilon,\varepsilon',\delta>0$.

For the non-uniqueness construction, we measure the amplitudes and frequencies of the building blocks by introducing the sequence
\begin{align*}
N_k\triangleq A^{b^k}, \qquad k\in\mathbb{N},
\end{align*}
where $A,b\in\mathbb{N}$. For the norm inflation construction, we introduce the parameters
\begin{align*}
    \lambda=\delta^{-1}\quad\mathrm{and}\quad \lambda_q=2^q,\quad q\in\mathbb{R}.
\end{align*}
We introduce a parameter $\gamma\in(0,\min\{\frac{\beta-1-\alpha}{2},\frac{1-2\alpha}{4}\}]$, which measures the subcriticality of the remaining terms. We further choose two parameters $\epsilon_1>0$ and $\epsilon_2>0$, where $\epsilon_1$ depends on $\epsilon_2$ and other universal constants appearing in the inequalities, and where $\epsilon_2\in(0,\frac{1-2\alpha-\gamma}{4}]$.
We also set a parameter $\tau\in(\frac{\beta+\alpha-1+\gamma}{2\beta},\frac{\beta-\alpha-\epsilon_2}{2\beta})$, which quantifies the subcriticality of the perturbations. 

Finally, the parameters $A$ and $q$ will be chosen sufficiently large, depending only on the other parameters.
We use the notation $X\lesssim Y$ to mean that $X\leq CY$ with some constants $C>0$ independent of A, but which may depend on other parameters. 
\subsection{Organization of the paper}
The organization of the rest of the paper is as follows. In Section \ref{sectprincipal}, we construct the initial data and the leading parts of the non-unique solutions, and show that the remaining terms are small in suitable subcritical norms. In Section \ref{sectioninlfation}, we construct the principal part of the solutions exhibiting norm inflation and also prove that the corresponding remaining terms are small in the same subcritical norms.
In Section \ref{sec:4}, we construct the perturbations associated with the vector fields constructed in the previous two sections. In Section \ref{Proof of the Non-Uniqueness}, we complete the proofs of Proposition \ref{mainprop1} and \ref{mainprop2}, thereby establishing the non-uniqueness and norm inflation results. Moreover, we prove the local well-posedness results, namely Theorem \ref{maintheo3} and \ref{maintheo4}.
The appendix contains some technical tools used in the paper.

\section{Construction of the principal part of solutions exhibiting non-uniqueness}\label{sectprincipal}

\subsection{Construction of the initial data} 

\begin{defi}\label{defiofS}
    We define
    \begin{itemize}
        \item A differential operator $S$ that takes vector fields to the symmetric rank-2
tensor,
\begin{gather*}
    S(f)\triangleq \begin{pmatrix}
2\partial_1f^1 & \partial_1f^2+\partial_2f^1 \\
\partial_1f^2+\partial_2f^1 & 2\partial_2f^2
\end{pmatrix}, \quad\forall f\in C^\infty(\mathbb{T}^2;\mathbb{R}^2),
\end{gather*}
Moreover, if $f$ is divergence-free, we have
\begin{gather}
    \mathrm{div}S(f)=\Delta f.\label{divS}
\end{gather}
\item A standard mollifier $\phi_k$ with length scale $N_{k+1}^{-\frac{1}{2}}N_{k}^{-\frac{1}{2}}$ is defined by
\begin{align*}
    \phi_k(x)=\sum_{l\in\mathbb{Z}^2}\widetilde{\phi}_k(x+l)\quad\mathrm{with}\quad \widetilde{\phi}_k(x)= N_{k+1}N_k\phi((N_{k+1}N_k)^{\frac{1}{2}}x),
\end{align*}
where $\phi\in C^\infty_x(\mathbb{R}^2;\mathbb{R})$ satisfying
\begin{align*}
    \phi\geq0,\quad \operatorname{supp}\phi\subset B_1(0),\quad  \int_{\mathbb{R}^2}\phi(x)\, dx=1.
\end{align*}
    \end{itemize}
\end{defi}

We construct the vector fields inductively. Recall that the set of vectors $\{\xi\}_{\xi\in\Lambda}$, the associated orthonormal vectors $\{\bar{\xi}\}_{\xi\in\Lambda}$ and the functions $\{a_\xi\}_{\xi\in\Lambda}$ are introduced in Lemma \ref{geometric}. 
\begin{defi}\label{Apropdefivk}
    Let $\xi$ be a vector in $\Lambda$. For the zeroth, we define
    \begin{gather}
        v_{0}^0(x)\triangleq N_0^{\beta+\alpha} \sin(2\pi N_0\xi\cdot x)\bar{\xi},
        \label{Adefiv0}
    \end{gather}
    Given a divergence-free $v_{k-1}^0$, we define the vector potential for the k-th case:
    \begin{gather}
        v_{k}^0(x)\triangleq C_\beta N_k^{-2[\beta]-1+\beta+\alpha}\Delta^{[\beta]}\nabla^{\perp}\left(\sum_{\xi\in\Lambda}a_{\xi,k}e^{2\pi iN_k\xi\cdot x} \right)*\phi_k(x),\label{defivk}
    \end{gather}
    where $C_\beta$ is a constant and $a_{\xi,k}$ is a function defined as
    \begin{align}
        C_\beta&\triangleq\sqrt{2}(2\pi)^{\beta-2[\beta]-1},\\
        a_{\xi,k}&\triangleq a_{\xi}\left(\frac{S(\Delta^{-1}v_{k-1}^0)}{N_k^{2\alpha}}+\mathrm{Id}\right).\label{defiaxik}
    \end{align}
\end{defi}

To verify that $\{v_{k}^0\}_{k\in\mathbb{N}}$ are well-defined, it suffices to establish that $\|\frac{S(\Delta^{-1}v_{k-1}^0)}{N_k^{2\alpha}}\|_{L^\infty({\mathbb{T}^2})}\leq \frac{1}{7}$ by Lemma \ref{geometric}.    

\begin{lemm}
     For all $k\in\mathbb{N}$, we have 
    \begin{gather}
        \|\frac{S(\Delta^{-1}v_{k}^0)}{N_{k+1}^{2\alpha}}\|_{L^\infty({\mathbb{T}^2})}\leq \frac{1}{7},\label{LinftyAphik2}
    \end{gather}
    provided that $A$ and $b$ are sufficiently large.
\end{lemm}

\begin{proof}
    With the standard mollifier estimate, it suffices to prove that
    \begin{align*}
        \|\frac{S(\Delta^{-1}\Phi_{k}^0)}{N_{k+1}^{2\alpha}}\|_{L^\infty({\mathbb{T}^2})}\leq \frac{1}{7},
    \end{align*}
    where $\Phi_{k}^0$ is the part of $v_{k}^0$ before applying the mollifier $\phi_k$:
    \begin{align}
        \Phi_{k}^0\triangleq\begin{cases}
            v_0^0,\quad k=0,\\
            C_\beta N_k^{-2[\beta]-1+\beta+\alpha}\Delta^{[\beta]}\nabla^{\perp}\left(\sum_{\xi\in\Lambda}a_{\xi,k}e^{2\pi iN_k\xi\cdot x} \right),\quad k\geq1.
        \end{cases}\label{AAAA}
    \end{align}
    
    We prove \eqref{LinftyAphik2} by induction. For k=0, it follows from \eqref{Adefiv0} and the fact that S is a first-order differential operator that 
    \begin{gather*}
        \| S(\Delta^{-1}\Phi_{0}^0)\|_{L^\infty({\mathbb{T}^2})}\lesssim N_0^{\beta+\alpha-1}.
    \end{gather*}
    Thus, choosing $A$ and $b$ sufficiently large, we get
     \begin{gather}
        \|\frac{S(\Delta^{-1}\Phi_{0}^0)}{N_{1}^{2\alpha}}\|_{L^\infty({\mathbb{T}^2})}\lesssim  N_0^{\beta+\alpha-1}N_1^{-2\alpha}\lesssim A^{\beta+\alpha-1-2b\alpha}  \leq \frac{1}{7}. 
    \end{gather}
    Once again from \eqref{Adefiv0}, it yields that
    \begin{gather}
        \|\nabla^m \frac{S(\Delta^{-1}v_{0}^0)}{N_{1}^{2\alpha}}\|_{L^\infty({\mathbb{T}^2})}\lesssim N_0^{m},\quad\quad \forall m\in\mathbb{N}.\label{mAphi0}
    \end{gather}
    Now, we suppose $\|\frac{S(\Delta^{-1}\Phi_{k-1}^0)}{N_{k}^{2\alpha}}\|_{L^\infty({\mathbb{T}^2})}\leq \frac{1}{7}$. Using the standard mollifier estimate or \eqref{mAphi0}, we have
    \begin{gather}
        \|\nabla^m\frac{S(\Delta^{-1}v_{k-1}^0)}{N_k^{2\alpha}}\|_{L^\infty({\mathbb{T}^2})}\lesssim (N_{k-1}N_k)^{\frac{m}{2}},\quad\quad \forall m\in\mathbb{N}. \label{mAphik-1}
    \end{gather}
    By virtue of \eqref{AAAA}, \eqref{mAphik-1} and the Leibniz rule, we deduce
    \begin{gather*}
         \| S(\Delta^{-1}\Phi_{k}^0)\|_{L^\infty({\mathbb{T}^2})}\lesssim  N_k^{-2[\beta]-1+\beta+\alpha}(N_k^{2[\beta]}+(N_{k-1}N_k)^{[\beta]})\lesssim N_k^{-1+\beta+\alpha},
    \end{gather*}
    which immediately implies 
    \begin{gather*}
        \|\frac{S(\Delta^{-1}\Phi_{k}^0)}{N_{k+1}^{2\alpha}}\|_{L^\infty({\mathbb{T}^2})}\lesssim N_k^{-1+\beta+\alpha} N_{k+1}^{-2\alpha}\lesssim A^{b^k(\beta+\alpha-1-2b\alpha)}\leq\frac{1}{7},
    \end{gather*}
    if $A$ and $b$ are sufficiently large.
\end{proof}

Next, we establish some estimates on $v^0_k$ and $a_{\xi,k}$.
\begin{prop}\label{propofinitial}
    For any $k\in\mathbb{N}$, $\{v^0_k\}_{k\in\mathbb{N}}$ is divergence-free and satisfies
    \begin{align}
        \|\nabla^m v^0_k\|_{L^{\infty}(\mathbb{T}^2)}&\lesssim N_k^{\beta+\alpha+m},\quad\forall m\in\mathbb{N}.\label{mvk}
    \end{align}
    Moreover, we have
    \begin{gather}
        \|\nabla^m a_{\xi,k}\|_{L^\infty(\mathbb{T}^2)}\lesssim N_{k-1}^m,\quad\forall m\in\mathbb{N}.\label{maxik}
    \end{gather}
\end{prop}
\begin{proof}
    The divergence-free property of $\{v^0_k\}_{k\in\mathbb{N}}$ is easily derived from Definition \ref{Apropdefivk}. With Definition \ref{Apropdefivk}, Definition \ref{defiV0} and \eqref{mAphik-1}, one can use the Leibniz rule to obtain 
    \begin{gather*}
        \|\nabla^m v^0_k\|_{L^\infty(\mathbb{T}^2)}\lesssim N_k^{-2[\beta]-1+\beta+\alpha}\left(N_k^{2[\beta]+1+m}+(N_{k-1}N_k)^{\frac{2[\beta]+1+m}{2}}\right)\lesssim N_k^{\beta+\alpha+m},
    \end{gather*} 
     and then applying the chain rule yields
    \begin{gather*}
        \|\nabla^m a_{\xi,k}\|_{L^\infty(\mathbb{T}^2)}\lesssim N_{k-1}^m,
    \end{gather*}
    for any $m\in\mathbb{N}$.
\end{proof}

Now, we define the initial data in terms of the vector fields introduced above.
\begin{defi}\label{defiV0}
   We set the initial data by
    \begin{gather*}
        V^0\triangleq\sum_{k\in\mathbb{N}}v_k^0.
    \end{gather*}
\end{defi}

Before establishing the smallness of $V^0$ in $B^{-\beta-\alpha-\varepsilon}_{\infty,\infty}(\mathbb{T}^2)$, for any $k\geq1$, we decompose $v^0_k\triangleq v^{0,p}_k+v^{0,e}_k$ where
      \begin{align}
          v^{0,p}_k&\triangleq(-1)^{[\beta]}\sqrt{2}(2\pi)^{\beta}iN_k^{\beta+\alpha}\sum_{\xi\in\Lambda}a_{\xi,k}e^{2\pi iN_k\xi\cdot x}\bar{\xi},\label{defiv0p}\\
          v^{0,e}_k&\triangleq v_k^{0,e,1}+v_k^{0,e,2},\nonumber
      \end{align}
      having defined
      \begin{align}
          v_k^{0,e,1}&\triangleq C_\beta N_k^{-2[\beta]-1+\beta+\alpha}\Delta^{[\beta]}\nabla^{\perp}\left(\sum_{\xi\in\Lambda}a_{\xi,k}e^{2\pi iN_k\xi\cdot x} \right)\nonumber\\&\quad-(-1)^{[\beta]}\sqrt{2}(2\pi)^{\beta}iN_k^{\beta+\alpha}\sum_{\xi\in\Lambda}a_{\xi,k}e^{2\pi iN_k\xi\cdot x}\bar{\xi},\label{defiv0e1}\\
          v_k^{0,e,2}&\triangleq C_\beta N_k^{-2[\beta]-1+\beta+\alpha}\Delta^{[\beta]}\nabla^{\perp}\left(\sum_{\xi\in\Lambda}a_{\xi,k}e^{2\pi iN_k\xi\cdot x} \right)*(\phi_k-\delta).\label{defiv0e2}
      \end{align}

Actually, $v^{0,p}_k$ denotes the component obtained from $v_k^0$ by applying all derivatives to the oscillatory factor $e^{2\pi iN_k\xi\cdot x}$, while omitting the mollifier $\phi_k$. On the other hand, $v_k^{0,e,1}$ denotes the component of $v_k^0$ in which not all derivatives fall on the oscillatory factor $e^{2\pi iN_k\xi\cdot x}$, again with the mollifier $\phi_k$ omitted. $v_k^{0,e,2}$ is the mollification error of $v_k^0$.

In the following, we establish the smallness of the initial data $V^0$ in ${B^{-\beta-\alpha-\varepsilon}_{\infty,\infty}(\mathbb{T}^2)}$.
\begin{prop}\label{propofinitial}
    Let $\varepsilon,\varepsilon'>0$.  $V^0$ is divergence-free and satisfies
    \begin{gather*}
        \|V^0\|_{B^{-\beta-\alpha-\varepsilon}_{\infty,\infty}(\mathbb{T}^2)}\leq \varepsilon'.
    \end{gather*}
\end{prop}
    
\begin{proof}
The divergence-free property of $V^0$ is inherited from $\{v^0_k\}_{k\in\mathbb{N}}$. We therefore omit the proof and focus on the estimates for $V^0$.
      
      First, we prove $\|\sum_{k\geq1}v^{0,p}_k\|_{B^{-\beta-\alpha-\varepsilon}_{\infty,\infty}(\mathbb{T}^2)}\lesssim\varepsilon'$. Without loss of generality, fix $\xi\in\Lambda$. It suffices to show $\|\sum_{k\geq1}N_k^{\beta+\alpha}a_{\xi,k}e^{2\pi iN_k\xi\cdot x}\|_{B^{-\beta-\alpha-\varepsilon}_{\infty,\infty}(\mathbb{T}^2)}\lesssim \varepsilon'$.
      Taking a fourier series of $\{a_{\xi,k}\}_{k\geq1}$, we have
      \begin{align}
          a_{\xi,k}(x)&=\sum_{l\in\mathbb{Z}^2}\hat{a}_{\xi,k}(l)e^{2\pi il\cdot x},\nonumber\\
          a_{\xi,k}(x)e^{2\pi iN_k\xi\cdot x}&=\sum_{l\in\mathbb{Z}^2}\hat{a}_{\xi,k}(l-N_k\xi)e^{2\pi il\cdot x}.\label{fourieraxi}
      \end{align}
     By \eqref{maxik} and Parseval's identity, we obtain
      \begin{gather*}
\|\nabla^ma_{\xi,k}\|_{L^2(\mathbb{T}^2)}=\left(\sum_{l\in\mathbb{Z}^d}|l|^{2m}|\hat{a}_{\xi,k}(l)|^2\right)^{\frac{1}{2}}\lesssim \|\nabla^ma_{\xi,k}\|_{L^\infty(\mathbb{T}^2)}\lesssim N_{k-1}^m,\quad\quad\forall m\in\mathbb{N},
      \end{gather*}
      which immediately implies
      \begin{gather}
          |l|^{m}|\hat{a}_{\xi,k}(l)|\lesssim N_{k-1}^m,\quad\quad\forall (l,m)\in \mathbb{Z}^2\times \mathbb{N}.\label{hataxik}
      \end{gather}

    Recalling the definition of $B^{-\beta-\alpha-\varepsilon}_{\infty,\infty}(\mathbb{T}^2)$ in Definition \ref{defibesov} and fixing $j\geq-1$, we decompose the following into two parts:
    \begin{align*}
        \Delta_j\sum_{k\in\mathbb{N}}N_k^{\beta+\alpha}a_{\xi,k}e^{2\pi iN_k\xi\cdot}&=\Delta_j(\sum_{k:|2^j-N_k|\geq\frac{3}{4}N_k}+\sum_{k:|2^j-N_k|\leq\frac{3}{4}N_k})N_k^{\beta+\alpha}a_{\xi,k}e^{2\pi iN_k\xi\cdot}\\
        &\triangleq I_1+I_2.
    \end{align*}

    Consider $I_1$. If $j\geq0$, by \eqref{fourieraxi} and \eqref{hataxik},
    \begin{align}
        2^{-j(\beta+\alpha+\varepsilon)}\|I_1\|_{L^{\infty}(\mathbb{T}^2)}&\lesssim \sum_{k:|2^j-N_k|\geq\frac{3}{4}N_k} 2^{-j(\beta+\alpha+\varepsilon)}N_k^{\beta+\alpha}\|\Delta_j(a_{\xi,k}e^{2\pi iN_k\xi\cdot})\|_{L^{\infty}(\mathbb{T}^2)}\nonumber\\
        &\lesssim \sum_{k:|2^j-N_k|\geq\frac{3}{4}N_k} 2^{-j(\beta+\alpha+\varepsilon)}N_k^{\beta+\alpha}\sum_{\frac{3}{4}\cdot2^j\leq|l|\leq\frac{8}{3}\cdot 2^j}|\hat{a}_{\xi,k}(l-N_k\xi)|\nonumber\\
        &\lesssim \sum_{k:|2^j-N_k|\geq\frac{3}{4}N_k} 2^{-j(\beta+\alpha+\varepsilon)}N_k^{\beta+\alpha}\sum_{\frac{3}{4}\cdot2^j\leq|l|\leq\frac{8}{3}\cdot 2^j} \frac{N_{k-1}^{\beta+\alpha+4}}{|l-N_k\xi|^{\beta+\alpha+4}}\nonumber
        \\&\lesssim\sum_{k\geq1}N_k^{\beta+\alpha}\sum_{l\in\mathbb{Z}^2\setminus\{0\}}\frac{N_{k-1}^{\beta+\alpha+4}}{|l|^3N_k^{\beta+\alpha+1}}\nonumber\\
        &\lesssim\sum_{k\geq1}\frac{N_{k-1}^{\beta+\alpha+4}}{N_k}\sum_{l\in\mathbb{Z}^2\setminus\{0\}}\frac{1}{|l|^3}\lesssim\varepsilon',\label{1estimateI1}
    \end{align}
    where we use the fact that $|l-N_k\xi|\gtrsim \max\{|l|,N_k\}$ holds provided $|2^j-N_k|\geq\frac{3}{4}N_k$ and $\frac{3}{4}\cdot2^j\leq|l|\leq\frac{8}{3}\cdot 2^j$. We also have $\sum_{k\geq1}\frac{N_{k-1}^{\beta+\alpha+4}}{N_k}\lesssim\varepsilon'$ if the parameters $A$ and $b$ are chosen sufficiently large.
Similarly, if $j=-1$, we have
\begin{align}
    2^{\beta+\alpha+\varepsilon}\|I_1\|_{L^{\infty}(\mathbb{T}^2)}&\lesssim\sum_{k:|\frac{1}{2}-N_k|\geq\frac{3}{4}N_k} 2^{\beta+\alpha+\varepsilon}N_k^{\beta+\alpha}\sum_{|l|\leq\frac{4}{3}}|\hat{a}_{\xi,k}(l-N_k\xi)|\nonumber\\
    &\lesssim\sum_{k\geq1} N_k^{\beta+\alpha}\sum_{|l|\leq\frac{4}{3}}\frac{N_{k-1}^{\beta+\alpha+1}}{|l-N_k\xi|^{\beta+\alpha+1}}\nonumber\\
    &\lesssim \sum_{k\geq1} \frac{N_{k-1}^{\beta+\alpha+1}}{N_k}\lesssim \varepsilon',\label{2estimateI1}
\end{align}
where we use the estimate $|l-N_k\xi|\gtrsim N_k$ that is satisfied with $|\frac{1}{2}-N_k|\geq\frac{3}{4}N_k$ and $|l|\leq\frac{4}{3}$.
    
    We now turn to the estimate of $I_2$. Since the bound $\|a_{\xi,k}e^{2\pi iN_k\xi\cdot}\|_{L^{\infty}(\mathbb{T}^2)}\lesssim1$, we deduce
    \begin{align*}
        2^{-j(\beta+\alpha+\varepsilon)}\|I_2\|_{L^{\infty}(\mathbb{T}^2)}&\lesssim \sum_{k:|2^j-N_k|\leq\frac{3}{4}N_k} 2^{-j(\beta+\alpha+\varepsilon)}N_k^{\beta+\alpha}\|\Delta_j(a_{\xi,k}e^{2\pi iN_k\xi\cdot})\|_{L^{\infty}(\mathbb{T}^2)}\\
        &\lesssim \sup_{m\geq1}\frac{1}{N_m^\varepsilon}\sum_{k:|2^j-N_k|\leq\frac{3}{4}N_k} \|a_{\xi,k}e^{2\pi iN_k\xi\cdot}\|_{L^{\infty}(\mathbb{T}^2)}\\
        &\lesssim \sup_{m\geq1}\frac{1}{N_m^\varepsilon}\#\{K\in\mathbb{N}:|2^j-N_k|\leq\frac{3}{4}N_k\}.
    \end{align*}
    With a large enough choice of $A$ and $b$, it holds $\sup_{m\geq1}\frac{1}{N_m^\varepsilon}\leq\varepsilon'$ and 
    \begin{align*}
        \#\{K\in\mathbb{N}:|2^j-N_k|\leq\frac{3}{4}N_k\}\leq \log_{A^b}\frac{4\cdot2^j}{\frac{4}{7}\cdot2^j}+1\leq 2,
    \end{align*}
    which give the bound
    \begin{align}
        2^{-j(\beta+\alpha+\varepsilon)}\|I_2\|_{L^{\infty}(\mathbb{T}^2)}\lesssim\varepsilon'.\label{estimateI2}
    \end{align}
    Combining \eqref{1estimateI1},\eqref{2estimateI1} and \eqref{estimateI2} together, we get
    \begin{gather}
        \|\sum_{k\geq1}v^{0,p}_k\|_{B^{-\beta-\alpha-\varepsilon}_{\infty,\infty}(\mathbb{T}^2)}\lesssim\varepsilon'.\label{vopk}
    \end{gather}

    By a similar argument, we obtain
    \begin{align}
        \|v^0_0\|_{B^{-\beta-\alpha-\varepsilon}_{\infty,\infty}(\mathbb{T}^2)}\lesssim\varepsilon',\label{v00}\\
        \|\sum_{k\geq1}v^{0,e,1}_k\|_{B^{-\beta-\alpha-\varepsilon}_{\infty,\infty}(\mathbb{T}^2)}\lesssim\varepsilon'.\label{0e1}
    \end{align}
     Since $v^{0,e,1}_k$ denotes the part where derivatives are not all applied to $e^{2\pi iN_k\xi\cdot x}$, the sum $\sum_{k\geq1}v^{0,e,1}_k$ indeed enjoys smallness in stronger spaces, such as $B^{-\beta-\alpha+1-\lambda}_{\infty,\infty}$ for any $0<\lambda<1$.

      We finish by estimating $\sum_{k\geq1}v^{0,e,2}_k$. Recall that the length scale of the mollifier $\phi_k$ is $(N_{k+1}N_k)^{\frac{1}{2}}$. Applying the standard mollifier estimate and \eqref{maxik}, we have
      \begin{align}
          \|\sum_{k\geq1}v^{0,e,2}_k\|_{B^{-\beta-\alpha-\varepsilon}_{\infty,\infty}(\mathbb{T}^2)}&\lesssim \sum_{k\geq1}\|v^{0,e,2}_k\|_{L^\infty(\mathbb{T}^2)}\nonumber\\&\lesssim \sum_{k\geq1} N_k^{-2[\beta]-1+\beta+\alpha}\|\Delta^{[\beta]}\nabla^{\perp}\left(\sum_{\xi\in\Lambda}a_{\xi,k}e^{2\pi iN_k\xi\cdot} \right)*(\phi_k-\delta)\|_{L^\infty(\mathbb{T}^2)}\nonumber\\
          &\lesssim\sum_{k\geq1} N_k^{-2[\beta]-1+\beta+\alpha}\|\nabla^{2[\beta]+2}\left(\sum_{\xi\in\Lambda}a_{\xi,k}e^{2\pi iN_k\xi\cdot} \right)\|_{L^\infty(\mathbb{T}^2)}(N_{k+1}N_k)^{-\frac{1}{2}}\nonumber\\
          &\lesssim \sum_{k\geq1} \frac{N_k^{\frac{1}{2}+\beta+\alpha}}{N_{k+1}^{\frac{1}{2}}}\lesssim \varepsilon',\label{0e2}
      \end{align}
      where $\sum_{k\geq1} \frac{N_k^{\frac{1}{2}+\beta+\alpha}}{N_{k+1}^{\frac{1}{2}}}\lesssim\varepsilon'$ provided the parameters $A$ and $b$ are chosen sufficiently large. 

      Combining \eqref{vopk}-\eqref{0e2}and, without loss of generality, taking $\varepsilon'$ smaller if necessary, we conclude that
      \begin{align*}
          \|V^0\|_{B^{-\beta-\alpha-\varepsilon}_{\infty,\infty}(\mathbb{T}^2)}&\leq\|\sum_{k\geq1}v^{0,p}_k\|_{B^{-\beta-\alpha-\varepsilon}_{\infty,\infty}(\mathbb{T}^2)}+ \|v^0_0\|_{B^{-\beta-\alpha-\varepsilon}_{\infty,\infty}(\mathbb{T}^2)}\\&\quad\,\,+
        \|\sum_{k\geq1}v^{0,e,1}_k\|_{B^{-\beta-\alpha-\varepsilon}_{\infty,\infty}(\mathbb{T}^2)}+\|\sum_{k\geq1}v^{0,e,2}_k\|_{B^{-\beta-\alpha-\varepsilon}_{\infty,\infty}(\mathbb{T}^2)}\\&\leq \varepsilon'.
      \end{align*}
\end{proof}

\subsection{Construction of the principal parts }
\begin{defi}\label{defiofvk}
    For any $k\in\mathbb{N}$, we define
    \begin{align*}
        v_k&\triangleq v_k^0 e^{-(2\pi N_k)^{2\beta}t},\\
        \bar{v}_k&\triangleq v_k^0 e^{-2\cdot(2\pi N_{k+1})^{2\beta}t}.
    \end{align*}
    Then we define the principal parts of two solutions of (\ref{e:NSf})
    \begin{gather}
    v^{(1)}\triangleq\sum_{k\geq0\,even}v_k+\sum_{k\geq0\,odd}\bar{v}_k,\\
    v^{(2)}\triangleq\sum_{k\geq0\,odd}v_k+\sum_{k\geq0\,even}\bar{v}_k.
    \end{gather}
\end{defi}

\begin{prop}\label{estimate of mvi}
    Let $0<p<\frac{2\beta}{\beta+\alpha}$. The two vector fields $v^{(1)}$ and $v^{(2)}$ are divergence-free, belong to $L^p([0,1];L^\infty(\mathbb{T}^2))$ and obey the following bounds
    \begin{gather}
        \|\nabla^m v^{(i)}(t)\|_{L^{\infty}(\mathbb{T}^2)}\lesssim_m t^{-\frac{\beta+\alpha+m}{2\beta}},\label{mvi}
    \end{gather} 
    for all $m\geq0$, $t>0$.
    Moreover, given any $\epsilon_1,\epsilon_2>0$, there exist sufficiently large constants A and b such that
    \begin{gather}
        \| v^{(i)}(t)\|_{L^{\infty}(\mathbb{T}^2)}\leq \epsilon_1 t^{-\frac{\beta+\alpha+\epsilon_2}{2\beta}}.\label{mvi2}
    \end{gather}
\end{prop}

\begin{proof}
   The divergence-free property of $v^{(1)}$ and $v^{(2)}$ is inherited from $v^0_k$. 
   
   Using \eqref{mvk}, we have
   \begin{align}
       \|\nabla^m v^{(i)}(t)\|_{L^{\infty}(\mathbb{T}^2)}\lesssim\sum_{k\in\mathbb{N}}N_k^{\beta+\alpha+m} e^{-N_k^{2\beta}t}.\label{nablamvi}
   \end{align}
   We analyze the series function on the right-hand side. As $N_k$ grows rapidly, the series is dominated by its largest term. For $t\leq N_0^{-2\beta}$, the dominant contribution occurs near $N_k=t^{-\frac{1}{2\beta}}$, so the series is bounded by $O(t^{-\frac{\beta+\alpha+m}{2\beta}})$. For $t>N_0^{-2\beta}$, the summands decrease monotonically, yielding the upper bound $N_0^{\beta+\alpha+m}e^{-N_0^{2\beta}t}=O_m(t^{-\frac{\beta+\alpha+m}{2\beta}}e^{-C_mN_0^{2\beta}t})$. Therefore, it yields \eqref{mvi}. Similarly, 
   \begin{gather*}
       \| v^{(i)}(t)\|_{L^{\infty}(\mathbb{T}^2)}\lesssim\sum_{k\in\mathbb{N}}N_k^{\beta+\alpha} e^{-N_k^{2\beta}t}\lesssim\sup_{k\in\mathbb{N}}\frac{1}{N_k^{\epsilon_2}}\sum_{k\in\mathbb{N}}N_k^{\beta+\alpha+\epsilon_2} e^{-N_k^{2\beta}t}\leq \epsilon_1 t^{-\frac{\beta+\alpha+\epsilon_2}{2\beta}}.
   \end{gather*}

Applying \eqref{mvk} again and using $0<p<\frac{2\beta}{\beta+\alpha}$, one infers
\begin{gather*}
    \|v^{(i)}\|_{L^p([0,1];L^\infty(\mathbb{T}^2))}\lesssim\sum_{k\in\mathbb{N}}N_k^{\beta+\alpha}\|e^{-N_k^{2\beta}t}\|_{L^p([0,1])}\lesssim\sum_{k\in\mathbb{N}}N_k^{\beta+\alpha-\frac{2\beta}{p}}\lesssim1.
\end{gather*}

\end{proof}

   \subsection{Estimates on the residual}
  For $i\in\{1,2\}$, we define the residual $F^{(i)}\in C^\infty((0,1]\times\mathbb{T}^2;S^{2\times2})$ in divergence form and $\widetilde{F}^{(i)}\in C^\infty((0,1]\times\mathbb{T}^2;\mathbb{R}^2)$ in Laplace form, satisfying
\begin{gather}
    -\mathbb{P} \operatorname{div}F^{(i)}-\Delta^{[\beta]}\widetilde{F}^{(i)}=\partial_tv^{(i)}+(-\Delta)^{\beta}v^{(i)}+\mathbb{P} \operatorname{div}(v^{(i)}\otimes v^{(i)}).\label{defiF}
\end{gather}

Before defining $F^{(i)}$ and $\widetilde{F}^{(i)}$, we first split $v_k$. Recalling \eqref{defiv0p}-\eqref{defiv0e2}, for any $k\geq1$, we decompose $v_k=v^p_k+v^{e}_k$ where
    \begin{align}
        v^p_k&\triangleq v^{0,p}_ke^{-(2\pi N_k)^{2\beta}t},\label{defivpk}\\
        v^{e}_k&\triangleq (v^{0,e,1}_k+v^{0,e,2}_k)e^{-(2\pi N_k)^{2\beta}t}.
    \end{align}
and for $k=0$, we let
\begin{align*}
    v_0^p=v_0 \quad \mathrm{and} \quad v^{e}_0=0.
\end{align*}
We then set the principal part of $v^{(1)}$ by 
\begin{align}
    v^p\triangleq\sum_{k\geq0,even}v^p_k,\label{defivp}
\end{align}
and the error part of $v^{(1)}$ by
\begin{align}
\widetilde{v}^e\triangleq\sum_{k\geq1\,odd}\bar{v}_k+\sum_{k\geq0\,even}v^{e}_k. \label{defitidleve}
\end{align}
Therefore, we can split $v^{(1)}$ as 
\begin{align}
    v^{(1)}=v^p+\widetilde{v}^e,\label{decomposev1}
\end{align}
and $v^{(2)}$ admits a similar decomposition.

We define $F^{(1)}$ and $\widetilde{F}^{(1)}$ explicitly. The definitions of $F^{(2)}$ and $\widetilde{F}^{(2)}$ are similar and will therefore be omitted.
   Recall the differential operator $S$ in Definition \ref{defiofS}. For any $k\in\mathbb{N}$, we define the tensor fields 
\begin{align*}
    \bar{R}_k\triangleq S(\Delta^{-1} v_k^0)e^{-2\cdot(2\pi N_{k+1})^{2\beta}t}.
\end{align*}
Thus, by \eqref{divS} and Definition \ref{defiofvk}, it holds
\begin{gather} \operatorname{div}\bar{R}_k=\bar{v}_k.\label{divbarR}
\end{gather}

Let us define the residual term $F^{(1)}$ satisfying
\begin{gather*}
    F^{(1)}\triangleq F^{(1)}_{1}+F^{(1)}_{2}+F^{(1)}_{3},
\end{gather*}
where
the linear dissipation error of $\bar{v}_k$:
\begin{align*}
    F^{(1)}_{1}\triangleq-\sum_{k\geq1,odd} (-\Delta)^{\beta}\bar{R}_k,
\end{align*}
the oscillation error:
\begin{align*}
    F^{(1)}_{2}\triangleq-\sum_{k\geq1,odd}\mathcal{R}\left(\partial_t\bar{v}_k+\mathbb{P}\operatorname{div}(v^p_{k+1}\otimes v^p_{k+1})\right),
\end{align*}
and the corrector error:
\begin{align*}
    F^{(1)}_{3}\triangleq-v^{(1)}\otimes v^{(1)}+\sum_{k\geq0,even}v^p_{k}\otimes v^p_{k}.
\end{align*}
Here, $\{\nabla p_k\}_{k\geq1}$ denote some pressure terms, which will be defined later. 

   Next, we define the linear error associated with $v_k$, denoted by $\widetilde{F}^{(1)}$
   \begin{gather}
       \widetilde{F}^{(1)}\triangleq\sum_{k\geq2,even}C_\beta N_k^{-2[\beta]-1+\beta+\alpha}\nabla^{\perp}\left(\sum_{\xi\in\Lambda}[a_{\xi,k},(-\Delta)^{\beta}]e^{2\pi iN_k\xi\cdot x} \right)*\phi_ke^{-(2\pi N_k)^{2\beta}t}.\label{defitidleF1}
   \end{gather}

Finally, we verify that $F^{(1)}$ and $\widetilde{F}^{(1)}$, defined above, satisfy \eqref{defiF}. A direct computation and \eqref{divbarR} show that
\begin{align}
     \mathbb{P}\operatorname{div}F^{(1)}_{1}=-(-\Delta)^{\beta}\sum_{k\geq1\,odd} \bar{v}_k.\label{J1}
\end{align}
Since $\operatorname{div}(v_0\otimes v_0)=0$, we have
\begin{align}
    \mathbb{P}\operatorname{div}F^{(1)}_{2}&=-\sum_{k\geq1,odd}\left(\partial_t \bar{v}_k+\mathbb{P}\operatorname{div}(v^p_{k+1}\otimes v^p_{k+1})\right)\nonumber\\&=-\partial_t\sum_{k\geq1,odd} \bar{v}_k-\mathbb{P}\operatorname{div}\left(\sum_{k\geq0,even}v^p_{k}\otimes v^p_{k}\right),
\end{align}
and
\begin{align}
    \mathbb{P}\operatorname{div}F^{(1)}_{3}=-\mathbb{P}\operatorname{div}(v^{(1)}\otimes v^{(1)})+\mathbb{P}\operatorname{div}\left(\sum_{k\geq0,even}v^p_{k}\otimes v^p_{k}\right).
\end{align}
Moreover, by Definition \ref{Apropdefivk} and Definition \ref{defiofvk}, together with the fact that $\{e^{2\pi iN_kx}e^{-(2\pi N_k)^{2\beta}t}\}_{k\geq1}$ and $v_0$ are solutions to the heat equation 
\begin{gather*}
    \left(\partial_t+(-\Delta)^{\beta}\right)u=0,
\end{gather*}
It follows that
\begin{align}
    \Delta^{[\beta]}\widetilde{F}^{(1)}=-\left(\partial_t+(-\Delta)^{\beta}\right)\sum_{k\geq0,even}v_k.\label{J3}
\end{align}
Combining \eqref{J1}-\eqref{J3} and recalling $v^{(1)}=\sum_{k\geq0,even}v_k+\sum_{k\geq1,odd}\bar{v}_k$, we conclude that \eqref{defiF} holds.

Recall the parameter $\gamma\in(0,\min\{\frac{\beta-1-\alpha}{2},\frac{1-2\alpha}{4}\}]$, which are introduced to measure the subcriticality of the remaining terms. In the following proposition, we prove that $\{F^{i}\}_{i\in\{0,1\}}$ and $\{\widetilde{F}_i\}_{i\in\{0,1\}}$ obey certain subcritical estimates.

\begin{prop}\label{Prop 4.1}
     Let $\eta>0$ and $i\in\{1,2\}$. For all $t\in(0,1]$, $F^{(i)}$ and $\widetilde{F}^{(i)}$ satisfy
    \begin{gather}
        \|F^{(i)}(t)\|_{L^\infty(\mathbb{T}^2)}\leq \eta t^{-\frac{2\beta+2\alpha-1+\gamma}{2\beta}},
    \end{gather}
    and
    \begin{align}
        \|\nabla^m\widetilde{F}^{(i)}(t)\|_{L^\infty(\mathbb{T}^2)}\leq \eta t^{-\frac{3\beta+\alpha-1+\gamma-2[\beta]+m}{2\beta}},\quad\forall m\in\{0,1\},
    \end{align}
    provided that $A$ and $b$ are sufficiently large. 
\end{prop}

\begin{proof}
    Recalling the definition of $F^{(1)}$ and $\widetilde{F}^{(1)}$ above, we proceed to verify that $F^{(1)}$ and $\widetilde{F}^{(1)}$ satisfy the desired estimates. The arguments for $F^{(2)}$ and $\widetilde{F}^{(2)}$ are analogous and will be omitted.

\textbf{Estimate of $\boldsymbol{F^{(1)}_{1}}$}. Applying Proposition \ref{propofinitial}, we calculate directly
\begin{align*}
     \|F^{(1)}_{1}(t)\|_{L^\infty(\mathbb{T}^2)}&\lesssim\sum_{k\geq1\,odd}N_k^{3\beta+\alpha-1}e^{-N_{k+1}^{2\beta}t}\\
     &\lesssim\sum_{k\in\mathbb{N}}N_{k}^{\frac{3\beta+\alpha-1}{b}}e^{-N_{k}^{2\beta}t}.
\end{align*}
Hence, we easily deduce
\begin{gather}
    \|F^{(1)}_{1}(t)\|_{L^\infty(\mathbb{T}^2)}\lesssim t^{-\frac{2\beta+2\alpha-1+\gamma}{2\beta}}\sup_{k\in\mathbb{N}}N_k^{\frac{3\beta+\alpha-1}{b}-(2\beta+2\alpha-1+\gamma)} \lesssim\eta t^{-\frac{2\beta+2\alpha-1+\gamma}{2\beta}},\label{F12}
\end{gather}
where $A$ and $b$ are chosen to be large enough.

\textbf{Estimate of $\boldsymbol{F^{(1)}_{2}}$}. Recalling \eqref{defiv0p} and \eqref{defivpk}, we have
\begin{align}
    \operatorname{div}(v^p_k\otimes v^p_k)=-\operatorname{div}\left(\sum_{\xi_1,\xi_2\in\Lambda}a_{\xi_1,k}a_{\xi_2,k}e^{2\pi iN_k(\xi_1+\xi_2)\cdot x}\bar{\xi}_1\otimes\bar{\xi}_2\right)2\cdot(2\pi N
    _k)^{2\beta}N_k^{2\alpha}e^{-2\cdot(2\pi N_k)^{2\beta}t}.\label{PP1}
\end{align}
Applying Lemma \ref{2dlemma} and the product rule, we have
\begin{align}
    &\quad\,\operatorname{div}\left(\sum_{\xi_1,\xi_2\in\Lambda}a_{\xi_1,k}a_{\xi_2,k}e^{2\pi iN_k(\xi_1+\xi_2)\cdot x}\bar{\xi}_1\otimes\bar{\xi}_2\right)\nonumber\\&=-\sum_{\xi\in\Lambda}\operatorname{div}(a_{\xi,k}^2\bar{\xi}\otimes\bar{\xi})+\sum_{\xi_1\neq-\xi_2\in\Lambda}\operatorname{div}(a_{\xi_1,k}a_{\xi_2,k}\bar{\xi}_1\otimes\bar{\xi}_2)e^{2\pi iN_k(\xi_1+\xi_2)x}\nonumber
    \\&\quad\,-\frac{1}{2}\nabla(|\sum_{\xi\in\Lambda}a_{\xi,k}e^{2\pi iN_k\xi\cdot x}\bar{\xi}|^2-|\sum_{\xi\in\Lambda}a_{\xi,k}e^{2\pi iN_k\xi\cdot x}|^2)\nonumber\\
    &\quad\,+\frac{1}{2}\sum_{\xi_1\neq\xi_2\in\Lambda}\nabla(a_{\xi_1,k}a_{\xi_2,k})(\bar{\xi}_1\cdot\bar{\xi}_2-1)e^{2\pi iN_k(\xi_1-\xi_2)x}.\label{PP2}
\end{align}
Moreover, using \eqref{defiaxik} and Lemma \ref{geometric}, we have
\begin{align}
    \sum_{\xi\in\Lambda}\operatorname{div}(a_{\xi,k}^2\bar{\xi}\otimes\bar{\xi})=\operatorname{div}(\frac{S(\Delta^{-1}v_{k-1}^0)}{N_k^{2\alpha}}+\mathrm{Id})=\frac{v^0_{k-1}}{N_k^{2\alpha}}.\label{Adiv1}
\end{align}
Therefore, with Definition \ref{defiofvk} and \eqref{Adiv1}, we have
\begin{gather}
\partial_t\bar{v}_k-\mathbb{P}\operatorname{div}\left(\sum_{\xi\in\Lambda}a_{\xi,k+1}^2\bar{\xi}\otimes\bar{\xi}\right)\cdot2(2\pi N_{k+1})^{2\beta}N_{k+1}^{2\alpha}e^{-2\cdot(2\pi N_{k+1})^{2\beta}t}=0.\label{PP3}
\end{gather}
Combining \eqref{PP1}-\eqref{PP3} together, we simplify $F^{(1)}_{2}$ as follows
\begin{align*}
    F^{(1)}_{2}&=-\sum_{k\geq1\,odd}\mathcal{R}\left(\partial_t\bar{v}_k+\mathbb{P}\operatorname{div}(v^p_{k+1}\otimes v^p_{k+1})\right)\\
    &=\mathcal{R}\mathbb{P}\Big(\sum_{k\geq1\,odd}\sum_{\xi_1\neq-\xi_2\in\Lambda}\operatorname{div}(a_{\xi_1,{k+1}}a_{\xi_2,{k+1}}\bar{\xi}_1\otimes\bar{\xi}_2)e^{2\pi iN_{k+1}(\xi_1+\xi_2)x}\\
    &\quad+\sum_{k\geq1\,odd}\sum_{\xi_1\neq\xi_2\in\Lambda}\frac{1}{2}\nabla(a_{\xi_1,{k+1}}a_{\xi_2,{k+1}})(\bar{\xi}_1\cdot\bar{\xi}_2-1)e^{2\pi iN_{k+1}(\xi_1-\xi_2)x}\Big)\\&\quad\quad\cdot 2(2\pi N_{k+1})^{2\beta}N_{k+1}^{2\alpha}e^{-2\cdot(2\pi N_{k+1})^{2\beta}t},
\end{align*}
where the pressure term is eliminated by the identity $\mathbb{P}\nabla p=0$.

 Suppose $\{b_k\}_{k\geq1}$ are smooth functions satisfying $\|\nabla^mb_k\|_{L^\infty}\lesssim N_{k-1}^m$ for any $m\in\mathbb{N}$ and Let $\xi\in\mathbb{Z}^2\setminus\{0\}$. Without loss of generality, it suffices to show that
 \begin{gather*}
     \|\sum_{k\geq1}\mathcal{R}\mathbb{P}(\nabla b_ke^{2\pi iN_k\xi\cdot x})N_k^{2\beta+2\alpha}e^{-(2\pi N_k)^{2\beta}t}\|_{L^{\infty}(\mathbb{T}^2)}\lesssim\eta t^{-\frac{2\beta+2\alpha-1+\gamma}{2\beta}}.
 \end{gather*}
By the boundedness of $\mathbb{P}$ on $C^\mu(\mathbb{T}^2)$ for $\mu\in(0,1)$ and \eqref{eR}, with $m$ chosen sufficiently large, we obtain, for any $k\geq1$,
\begin{gather*}
    \|\mathcal{R}\mathbb{P}(\nabla b_ke^{2\pi iN_k\xi\cdot x})\|_{C^\mu(\mathbb{T}^2)}\lesssim N_k^{-1+\mu}N_{k-1}+N_k^{-m+\mu}N_{k-1}^{m+1}+N_k^{-m}N_{k-1}^{m+\mu+1}\lesssim N_k^{-1+\mu}N_{k-1}.
\end{gather*}
Thus, setting $\mu=\frac{\gamma}{2}$ and choosing $A,b$ sufficiently large, we arrive at
\begin{align}
    \|\sum_{k\geq1}\mathcal{R}\mathbb{P}(\nabla b_ke^{2\pi iN_k\xi\cdot x})N_k^{2\beta+2\alpha}e^{-(2\pi N_k)^{2\beta}t}\|_{C^\mu(\mathbb{T}^2)}&\lesssim\sum_{k\geq1}N_k^{2\beta+2\alpha-1+\mu}N_{k-1}e^{-(N_k)^{2\beta}t}\nonumber\\
    &\lesssim \sup_{k}N_k^{\frac{1}{b}-\frac{\gamma}{2}}t^{-\frac{2\beta+2\alpha-1+\gamma}{2\beta}}\nonumber\\
    &\lesssim \eta t^{-\frac{2\beta+2\alpha-1+\gamma}{2\beta}},\label{F131}
\end{align}
where we first choose $b$ sufficiently large again so that $\frac{1}{b}-\frac{\gamma}{2}<0$ and then take $A$ sufficiently large. By \eqref{F131}, we get the desired result 
\begin{gather}
    \|F^{(1)}_{2}(t)\|_{L^{\infty}(\mathbb{T}^2)}\lesssim\eta t^{-\frac{2\beta+2\alpha-1+\gamma}{2\beta}}.\label{AF13}
\end{gather}

\textbf{Estimate of $\boldsymbol{F^{(1)}_{3}}$}. Recall the definition of $v^p_k$, $v^{e}_k$, $v^p$ and $\widetilde{v}^e$ in \eqref{defivpk}-\eqref{defitidleve}. 
Using \eqref{maxik}, the standard mollifier estimate and the Leibniz rule, we have
\begin{align}
    \| v^p_k\|_{_{L^\infty(\mathbb{T}^2)}}&\lesssim N_k^{\beta+\alpha}e^{-N_k^{2\beta}t},\label{mvpk}\\
     \| v^{e}_k\|_{_{L^\infty(\mathbb{T}^2)}}&\lesssim (\frac{N_{k-1}}{N_k}+\frac{N_{k}^{\frac{1}{2}}}{N_{k+1}^{\frac{1}{2}}})N_k^{\beta+\alpha}e^{-N_k^{2\beta}t},\label{mvek}\\
     \| \bar{v}_k\|_{_{L^\infty(\mathbb{T}^2)}}&\lesssim N_k^{\beta+\alpha}e^{-N_{k+1}^{2\beta}t},\label{mbarvk}
\end{align}
for all $k\in\mathbb{N}$. 

 Due to \eqref{mvpk}, we know that $v^p$ admits the same estimate as $v^{(1)}$,
 \begin{gather}
     \| v^p(t)\|_{L^\infty(\mathbb{T}^2)}\lesssim t^{-\frac{\beta+\alpha}{2\beta}}.\label{mvp}
 \end{gather}
 \eqref{mvek} and \eqref{mbarvk} imply
\begin{align*}
    \|\widetilde{v}^e(t)\|_{L^\infty(\mathbb{T}^2)}&\lesssim \sum_{k\geq1}(\frac{N_{k-1}}{N_k}+\frac{N_{k}^{\frac{1}{2}}}{N_{k+1}^{\frac{1}{2}}})N_k^{\beta+\alpha}e^{-N_k^{2\beta}t}+N_k^{\beta+\alpha}e^{-N_{k+1}^{2\beta}t}\\
    &\lesssim\sum_{k\in\mathbb{N}}
    \left(N_k^{1/b-1}+N_k^{1/2-b/2}+N_k^{(\beta+\alpha)(\frac{1}{b}-1)}\right)N_k^{\beta+\alpha}e^{-N_k^{2\beta}t}\\
    &\lesssim \sum_{k\in\mathbb{N}} N_k^{1/b-1}N_k^{\beta+\alpha}e^{-N_k^{2\beta}t}.
\end{align*}
 Therefore, we deduce
 \begin{align}
     \|\widetilde{v}^e(t)\|_{L^\infty(\mathbb{T}^2)}\lesssim t^{-\frac{\beta+\alpha-1+\gamma}{2\beta}}\sup_{k\in\mathbb{N}}{N_k^{1/b-\gamma}}\lesssim\eta t^{-\frac{\beta+\alpha-1+\gamma}{2\beta}},\label{widevx}
  \end{align}
where we first take $b$ sufficiently large such that $1/b-\gamma<0$ and then choose $A$ large enough.

In view of the decomposition of $v^{(1)}$ in \eqref{decomposev1}, we spilt $F^{(1)}_{3}$ as
\begin{align*}
    F^{(1)}_{3}&=-v^{(1)}\otimes v^{(1)}+\sum_{k\geq0\,even}v^p_{k}\otimes v^p_{k}\\&=-\sum_{k\neq j\geq0\,even}v^p_{k}\otimes v^p_{j}-v^{(1)}\otimes\widetilde{v}^e-\widetilde{v}^e\otimes v^p.
\end{align*}
Applying \eqref{mvpk}, we obtain
\begin{align*}
    \|\sum_{k\neq j}v^p_{k}\otimes v^p_{j}\|_{L^\infty(\mathbb{T}^2)}&\lesssim\sum_{k\geq 1,j<k}\| (v^p_{k}\otimes v^p_{j})\|_{L^\infty(\mathbb{T}^2)}\\
     &\lesssim\sum_{k\geq1}N_k^{\beta+\alpha}e^{-N_k^{2\beta}t}\sum_{j<k}N_j^{\beta+\alpha}\\
     &\lesssim\sum_{k\geq1}N_k^{\beta+\alpha}N_{k-1}e^{-N_k^{2\beta}t}
     \\ &\lesssim\sum_{k\geq1}N_k^{(\beta+\alpha)(\frac{1}{b}+1)}e^{-N_k^{2\beta}t}.
\end{align*}
Thus, with $A,b$ chosen large enough, it follows 
\begin{align}
    \|\sum_{k\neq j}v^p_{k}\otimes v^p_{j}\|_{L^\infty(\mathbb{T}^2)}\lesssim t^{-\frac{2\beta+\alpha-1+\gamma}{2\beta}}\sup_{k\in\mathbb{N}}N_k^{(\beta+\alpha)(\frac{1}{b}+1)-(2\beta+2\alpha-1+\gamma)}
    \lesssim\eta t^{-\frac{2\beta+2\alpha-1+\gamma}{2\beta}},\label{F41}
\end{align}
where we have used the fact that $\beta>1$.
By virtue of \eqref{mvi}, \eqref{mvp} and \eqref{widevx}, we deduce
\begin{align}
    \|v^{(1)}\otimes\widetilde{v}^e+\widetilde{v}^e\otimes v^p\|_{L^{\infty}(\mathbb{T}^2)}&\lesssim \eta t^{-\frac{\beta+\alpha-1+\gamma}{2\beta}}\cdot t^{-\frac{\beta+\alpha}{2\beta}}\lesssim \eta t^{-\frac{2\beta+2\alpha-1+\gamma}{2\beta}},\label{F42}
\end{align}
Finally, \eqref{F41} and \eqref{F42} yield
\begin{gather}
    \|F^{(1)}_{3}(t)\|_{L^{\infty}(\mathbb{T}^2)}\lesssim\eta t^{-\frac{2\beta+2\alpha-1+\gamma}{2\beta}}.\label{F14}
\end{gather}

\textbf{Estimate of $\boldsymbol{F^{(1)}}$}.
Combining \eqref{F12}, \eqref{AF13} and \eqref{F14}, and taking a new $\eta$ small enough to absorb all constants, we conclude that
\begin{gather*}
    \|F^{(1)}(t)\|_{L^{\infty}(\mathbb{T}^2)}\leq \eta t^{-\frac{2\beta+2\alpha-1+\gamma}{2\beta}}.
\end{gather*}

\textbf{Estimate of $\boldsymbol{\widetilde{F}^{(1)}}$}. Recall the definition of $\widetilde{F}^{(1)}$ in \eqref{defitidleF1},
\begin{align*}
   \widetilde{F}^{(1)}=\sum_{k\geq2,even}C_\beta N_k^{-2[\beta]-1+\beta+\alpha}\nabla^{\perp}\left(\sum_{\xi\in\Lambda}[a_{\xi,k},(-\Delta)^{\beta}]e^{2\pi iN_k\xi\cdot x} \right)*\phi_ke^{-(2\pi N_k)^{2\beta}t}
\end{align*}
 Thanks to \eqref{maxik}, \eqref{commuator} and the standard mollifier estimates, for $m\in\{1,2\}$, one infers 
\begin{align*}
    \|\nabla^m\widetilde{F}^{(1)}(t)\|_{L^\infty(\mathbb{T}^2)}&\lesssim\sum_{k\geq1}\sum_{\xi\in\Lambda}
    N_k^{-2[\beta]-1+\beta+\alpha}\|\nabla^{1+m}[(-\Delta)^{\beta}, a_{\xi,k}]e^{2\pi iN_k\xi\cdot}\|_{L^\infty(\mathbb{T}^2)}e^{-(2\pi N_k)^{2\beta}t}\\
    &\lesssim\sum_{k\geq1}N_k^{-2[\beta]-1+\beta+\alpha}(N_k+N_{k-1})^{1+m}\left(N_k^{2\beta-1}N_{k-1}^4+N_{k-1}^{2\beta+3}\right)e^{- N_k^{2\beta}t}\\
     &\lesssim\sum_{k\in\mathbb{N}}N_k^{3\beta+\alpha-1-2[\beta]+\frac{4}{b}+m}e^{- N_k^{2\beta}t}.
\end{align*}
Therefore, for $m\in\{1,2\}$, we obtain
\begin{gather*}
     \|\nabla^m\widetilde{F}^{(1)}(t)\|_{L^\infty(\mathbb{T}^2)}\lesssim\sup_{k}N_k^{\frac{4}{b}-\gamma}t^{-\frac{3\beta+\alpha-1+\gamma-2[\beta]+m}{2\beta}}\leq\eta t^{-\frac{3\beta+\alpha-1+\gamma-2[\beta]+m}{2\beta}},
\end{gather*}
provided $A,b$ sufficiently large.
\end{proof}

\section{Construction of the principal part of solutions exhibiting norm inflation}\label{sectioninlfation}
Recall that $\delta>0$ and
\begin{align*}
    \lambda=\delta^{-1},\qquad \lambda_q=2^q,
\end{align*}
where $q\in\mathbb{R}$. We define the principal part of the norm inflation solution as follows.
\begin{defi}\label{propdefiU1}
We define the velocity fields $U_1$ and $U_2$ by: 
\begin{align}
U_1(x,t) &\triangleq \Phi_1(1-e^{-2\cdot (2\pi \lambda_q)^{2\beta}t}),
\label{U_0}
\\
U_2(x,t) &\triangleq C_\beta{\lambda_q}^{\beta+\alpha-2[\beta]-1}\Delta^{[\beta]}\nabla^{\perp}\!\left(\sum_{\xi}\bar{a}_\xi e^{2\pi i \lambda_q\xi\cdot x}\right)e^{-(2\pi \lambda_q)^{2\beta}t},
\label{U_1}
\end{align}
with
\begin{align*}
    \Phi_1&\triangleq \lambda^{s+1}\sin(2\pi \lambda\xi\cdot x)\bar{\xi}e^{-(2\pi \lambda)^{2\beta}t},\\
    C_\beta&\triangleq \sqrt{2}(2\pi)^{\beta-2[\beta]-1},\\
    \bar{a}_\xi&\triangleq a_{\xi}\left(-\frac{S(\Delta^{-1}\Phi_1)}{ \lambda_q^{2\alpha}}+\mathrm{Id}\right).
\end{align*}
And we define the principal part of the solution by
\begin{align*}
    U(x,t)=U_1(x,t)+U_2(x,t).
\end{align*}
\end{defi}

\begin{prop}\label{low bound of U} 
The vector field $U$ is divergence-free. Moreover, provided that $q$ is sufficiently large, for any $n,m\in\mathbb{N}$, the function $\bar{a}_\xi$ satisfies
\begin{align*}
    \|\partial_t^n\nabla^m\bar{a}_\xi\|_{L^\infty_{t,x}(\mathbb{T}^2)}&\lesssim \lambda^{2\beta n+m}
\end{align*}    
and the two frequency-localized components $U_1$ and $U_2$ satisfy
\begin{align}
        \|\nabla^mU_1(t)\|_{L^\infty(\mathbb{T}^2)}&\lesssim \lambda^{s+1+m}e^{- \lambda^{2\beta}t},\label{estimateU1}\\
        \|\nabla^mU_2(t)\|_{L^\infty(\mathbb{T}^2)}&\lesssim \lambda_q^{\beta+\alpha+m}e^{- \lambda_q^{2\beta}t}.\nonumber
    \end{align}
\end{prop}
\begin{proof}
    The proof is analogous to that of Proposition \ref{estimate of mvi}, and is therefore omitted.
\end{proof}
We now define the residuals \(G\) and \(\widetilde{G}\) through the identity, 
\begin{align}\label{G about U} 
        -\mathbb{P}\operatorname{div}G-\Delta^{[\beta]}\widetilde{G}=\partial_t U+(-\Delta)^{\beta} U+ \mathbb{P}\operatorname{div} (U\otimes  U),  
\end{align}
where \(G\in C^{\infty}((0,1]\times\mathbb{T}^2;S^{2\times2})\) is a matrix-valued residual in divergence form, while
\(\widetilde{G}\in C^{\infty}((0,1]\times\mathbb{T}^2;\mathbb{R}^2)\) is a vector-valued residual in Laplacian form.
Before defining $G$ and $\widetilde{G}$, we first decompose $U_2=U_2^p+U_2^e,$
with
\begin{align}
    &U_2^p\triangleq \sqrt{2}(2\pi)^{\beta}{\lambda_q}^{\beta+\alpha}i\sum_{\xi}\bar{a}_\xi e^{2\pi i \lambda_q\xi\cdot x}\bar{\xi}e^{-(2\pi \lambda_q)^{2\beta}t},\label{U_2p}
    \\&U_2^e\triangleq C_\beta{\lambda_q}^{\beta+\alpha-2[\beta]-1}\Delta^{[\beta]}\nabla^{\perp}\!\left(\sum_{\xi}\bar{a}_\xi e^{2\pi i \lambda_q\xi\cdot x}\right)e^{-(2\pi \lambda_q)^{2\beta}t}\nonumber
    \\&\quad\quad-\sqrt{2}(2\pi)^{\beta}{\lambda_q}^{\beta+\alpha}i\sum_{\xi}\bar{a}_\xi e^{2\pi i \lambda_q\xi\cdot x}\bar{\xi}e^{-(2\pi \lambda_q)^{2\beta}t},\label{U_2e}
\end{align}
We are now in a position to define $G$ and $\widetilde{G}$. 
      Let us define the residual term $G$ by decomposing it into two parts:
\begin{align*}
    G\triangleq G_1+G_2,
\end{align*}
    where $G_1$ is the oscillation error:
    \begin{align*}
        G_1\triangleq -\mathcal{R}\left(\mathbb{P}\mathrm{div}(U_2^p\otimes U_2^p)+(\partial_t+(-\Delta)^{\beta})U_1\right),
    \end{align*}
    and $G_2$ is the corrector error:
    \begin{align*}
        G_2\triangleq-U^e_2\otimes U_2-U^p_2\otimes U_2^e.
    \end{align*}
    Here $p$ denotes the pressure terms, which will be define later.  

    Next, we define the linear error $\widetilde{G}$,
    \begin{align}\label{g}
        \widetilde{G}\triangleq&C_\beta{\lambda_q}^{\beta+\alpha-2[\beta]-1}\nabla^{\perp}\!\left(\sum_{\xi}[\bar{a}_\xi,(-\Delta)^{\beta}] e^{2\pi i \lambda_q\xi\cdot x}\right)e^{-(2\pi \lambda_q)^{2\beta}t} \nonumber
        \\&-
        C_\beta{\lambda_q}^{\beta+\alpha-2[\beta]-1}\nabla^{\perp}\!\left(\sum_{\xi}\partial_t\bar{a}_\xi\cdot e^{2\pi i \lambda_q\xi\cdot x}\right)e^{-(2\pi \lambda_q)^{2\beta}t}
        .
    \end{align}
    
Finally, we verify that $G$ and $\widetilde{G}$. As $\operatorname{div}\mathcal{R}=Id$, we then get
\begin{align}\label{G_1}
    \mathbb{P}\operatorname{div} G = -\mathbb{P}\operatorname{div}(U\otimes U)-(\partial_t+(-\Delta)^{\beta})U_1.
\end{align}
Moreover, using the fact that $e^{2\pi i\lambda_q\xi \cdot x}e^{-(2\pi \lambda_q)^{2\beta}t}$ are solutions to the fractional heat equation $(\partial_t+(-\Delta)^{\beta})u=0$, we derive
\begin{align}\label{G_3}
    \Delta^{[\beta]}\widetilde{G}=-(\partial_t+(-\Delta)^{\beta})U_2.
\end{align}
It follows from \eqref{G_1}, \eqref{G_3} and the identity $U=U_1+U_2$ that \eqref{G about U} holds.   
\begin{prop}\label{Norm flation G}
    Let $\eta>0$. For all $t\in(0,1]$ and $m\in\{0,1\}$, $G$ and $\widetilde{G}$ satisfy
    \begin{gather}
        \|G(t)\|_{L^\infty(\mathbb{T}^2)}\leq \eta t^{-\frac{2\beta+2\alpha-1+\gamma}{2\beta}},
    \end{gather}
    and
    \begin{align}
        \|\nabla^m\widetilde{G}(t)\|_{L^\infty(\mathbb{T}^2)}\leq \eta t^{-\frac{3\beta+\alpha-1+\gamma-2[\beta]+m}{2\beta}},
    \end{align}
    provided that $q$ is sufficiently large. 
\end{prop}
\begin{proof}  
  
\textbf{Estimate of $\boldsymbol{G_{1}}$}.
By the definition of $U^p_2$ in \eqref{U_2p}, we have
\begin{align}\label{G_1 1}
&\operatorname{div}\bigl(U_2^p\otimes U_2^p\bigr) = -
\operatorname{div}\left(
\sum_{\xi_1,\xi_2\in\Lambda}
\bar{a}_{\xi_1}\bar{a}_{\xi_2}
e^{2\pi i\lambda_q(\xi_1+\xi_2)\cdot x}\,\bar{\xi}_1\otimes\bar{\xi}_2
\right)2\cdot(2\pi)^{2\beta}{\lambda_q}^{2\beta+2\alpha}e^{-2(2\pi \lambda_q)^{2\beta}t}
\end{align}
Using Lemma \ref{2dlemma} and the product rule, one gets
\begin{align}\label{G_1 2}
&\operatorname{div}\left(
\sum_{\xi_1,\xi_2\in\Lambda}
\bar{a}_{\xi_1}\bar{a}_{\xi_2}
e^{2\pi i\lambda_q(\xi_1+\xi_2)\cdot x}\,\bar{\xi}_1\otimes\bar{\xi}_2
\right)\nonumber \\
&=-\sum_{\xi\in\Lambda}\operatorname{div}\left(\bar{a}_{\xi}^{2}\bar{\xi}\otimes\bar{\xi}\right)
+\sum_{\xi_1\neq\xi_2\in\Lambda}\operatorname{div}\Big(
\bar{a}_{\xi_1}\bar{a}_{\xi_2}\bar{\xi}_1\otimes\bar{\xi}_2
\Big)e^{2\pi i\lambda_q(\xi_1+\xi_2)\cdot x} \nonumber  
\\&\quad-\frac12\nabla\left(
\left|\sum_{\xi\in\Lambda}\bar{a}_{\xi}e^{2\pi i\lambda_q\xi\cdot x}\bar{\xi}\right|^2
-\left|\sum_{\xi\in\Lambda}\bar{a}_{\xi}e^{2\pi i\lambda_q\xi\cdot x}\right|^2\right)  +\frac12\sum_{\xi_1\neq\xi_2\in\Lambda}\nabla\left(
\bar{a}_{\xi_1}
\bar{a}_{\xi_2}
\right)\bigl(\bar{\xi}_1\cdot\bar{\xi}_2-1\bigr)e^{2\pi i\lambda_q(\xi_1-\xi_2)\cdot x}.
\end{align}
where the first term corresponds to the low-frequency contribution.
Furthermore, by Lemma \ref{geometric}, we obtain
\begin{align*}
\sum_{\xi\in\Lambda}\operatorname{div}\left(\bar{a}_{\xi}^{2}\bar{\xi}\otimes\bar{\xi}\right)
=\operatorname{div}\left(-\frac{S(\Delta^{-1}\Phi_1)}{ \lambda_q^{2\alpha}}+\mathrm{Id}\right)
=-\frac{\Phi_1}{\lambda_q^{2\alpha}}.
\end{align*}
Therefore, using the fact that $\Phi_1$ is a solution of heat equation, we have 
\begin{align}\label{G_1 3}
(\partial_t+(-\Delta)^{\beta})U_1
+\mathbb{P}\operatorname{div}\sum_{\xi\in\Lambda}\bar{a}_{\xi}^{2}\bar{\xi}\otimes\bar{\xi}
\cdot 2(2\pi)^{2\beta}{\lambda_q}^{2\beta+2\alpha}e^{-2(2\pi \lambda_q)^{2\beta}t}=0.
\end{align}
Combining \eqref{G_1 1}-\eqref{G_1 3} and the fact that $\mathbb{P}\nabla=0$, we may rewrite $G_1$ as
\begin{align*}
    G_1&=-\mathcal{R}\left(\mathbb{P}\mathrm{div}(U_2^p\otimes U_2^p)+(\partial_t+(-\Delta)^{\beta})U_1\right)\\
    &=\mathcal{R}\mathbb{P}\Big(\sum_{\xi_1\neq\xi_2\in\Lambda}\operatorname{div}\Big(
\bar{a}_{\xi_1}\bar{a}_{\xi_2}\bar{\xi}_1\otimes\bar{\xi}_2
\Big)e^{2\pi i\lambda_q(\xi_1+\xi_2)\cdot x} \nonumber  +\frac12\sum_{\xi_1\neq\xi_2\in\Lambda}\nabla\left(
\bar{a}_{\xi_1}
\bar{a}_{\xi_2}
\right)\bigl(\bar{\xi}_1\cdot\bar{\xi}_2-1\bigr)e^{2\pi i\lambda_q(\xi_1-\xi_2)\cdot x}\Big)\\&\quad\quad\cdot 2(2\pi)^{2\beta}{\lambda_q}^{2\beta+2\alpha}e^{-2(2\pi \lambda_q)^{2\beta}t},
\end{align*}

 Let $b$ is a smooth function with $\|\nabla^m b\|_{L^\infty}\lesssim \lambda^m$ for any $m\in\mathbb{N}$ and Let $\xi\in\mathbb{Z}^2\setminus\{0\}$. Thus, it suffices to show that
 \begin{gather*}
     \|\mathcal{R}\mathbb{P}(\nabla b e^{2\pi i\lambda_q\xi\cdot x}){\lambda_q}^{2\beta+2\alpha}e^{-(2\pi \lambda_q)^{2\beta}t}\|_{L^{\infty}(\mathbb{T}^2)}
     \lesssim \eta t^{-\frac{2\beta+2\alpha-1+\gamma}{2\beta}}.
 \end{gather*}
By the boundedness of $\mathbb{P}$ on $C^{\mu}(\mathbb{T}^2)$ for $\mu\in(0,1)$ and by applying \eqref{eR} with $m$ chosen large enough, we obtain
\begin{gather*}
    \|\mathcal{R}\mathbb{P}(\nabla be^{2\pi i\lambda_q\xi\cdot x})\|_{C^\mu(\mathbb{T}^2)}\lesssim \lambda_q^{-1+\mu}\lambda+\lambda_q^{-m+\mu}\lambda^{m+1}+\lambda_q^{-m}\lambda^{m+\mu+1}\lesssim \lambda_q^{-1+\mu}\lambda.
\end{gather*}
Then choosing $\mu=\frac{\gamma}{2}$ and taking $q$ sufficiently large, we arrive at
\begin{align}
    \|\mathcal{R}\mathbb{P}(\nabla b e^{2\pi i\lambda_q\xi\cdot x}){\lambda_q}^{2\beta+2\alpha}e^{-(2\pi \lambda_q)^{2\beta}t}\|_{C^\mu(\mathbb{T}^2)}
    &\lesssim
    \lambda_q^{2\beta+2\alpha-1+\mu}\lambda e^{-(2\pi \lambda_q)^{2\beta}t}\nonumber\\
    &\lesssim 
    \lambda_q^{-\frac{\gamma}{2}}\lambda_q^{2\beta+2\alpha-1+\gamma} e^{-(2\pi \lambda_q)^{2\beta}t}\nonumber\\
    &\lesssim \eta t^{-\frac{2\beta+2\alpha-1+\gamma}{2\beta}}.\label{G131}
\end{align}
Consequently,
\begin{gather}
    \|G_{1}\|_{L^{\infty}(\mathbb{T}^2)}\lesssim\eta t^{-\frac{2\beta+2\alpha-1+\gamma}{2\beta}}.\label{G1 end}
\end{gather}
\textbf{Estimate of $\boldsymbol{G_{2}}$}.
By the definition of $U^e_2$ and $U^p_2$ in \eqref{U_2p} and \eqref{U_2e}, respectively, together with the Leibniz rule, we obtain
\begin{align*}
    &\|U_2^p\|_{L^{\infty}(\mathbb{T}^2)} 
    \lesssim \lambda_q^{\beta+\alpha}e^{-(2\pi \lambda_q)^{2\beta}t},
    \\&\|U_2^e\|_{L^{\infty}(\mathbb{T}^2)} 
    \lesssim \lambda\lambda_q^{\beta+\alpha-1}e^{-(2\pi \lambda_q)^{2\beta}t}.
\end{align*}
provided $q$ is sufficiently large. Hence, for sufficiently large $q$, we get
\begin{align}\label{G_2 1}
    \|U_2^e\otimes U_2^p\|_{L^{\infty}(\mathbb{T}^2)} 
    &\lesssim \lambda_q^{2\beta+2\alpha-1}e^{-2(2\pi \lambda_q)^{2\beta}t}
    \nonumber\\&\lesssim \lambda_q^{-\gamma}t^{-\frac{2\beta+2\alpha-1+\gamma}{2\beta}}
    \nonumber\\&\lesssim \eta t^{-\frac{2\beta+2\alpha-1+\gamma}{2\beta}},
\end{align}
Similarly, it holds
\begin{align}\label{G_2 3}
    \|U_2^e\otimes U_2^e\|_{L^{\infty}(\mathbb{T}^2)} 
    &\lesssim \eta t^{-\frac{2\beta+2\alpha-1+\gamma}{2\beta}}.
\end{align}
It follows from \eqref{G_2 1}, \eqref{G_2 3} and the definition of $G_2$ that
\begin{gather}
    \|G_{2}\|_{L^{\infty}(\mathbb{T}^2)}\lesssim\eta t^{-\frac{2\beta+2\alpha-1+\gamma}{2\beta}}.\label{G2 end}
\end{gather}
\textbf{Estimate of $\boldsymbol{G}$}.
By combining \eqref{G1 end} and \eqref{G2 end} and choosing a new $\eta$ small enough to absorb all involved constants, one gets
\begin{align*}
    \|G\|_{L^{\infty}(\mathbb{T}^2)}\leq\eta t^{-\frac{2\beta+2\alpha-1+\gamma}{2\beta}}.
\end{align*}
\textbf{Estimate of $\boldsymbol{\widetilde{G}}$}.
We split $\widetilde{G}$ as $\widetilde{G}=\widetilde{G}_1+\widetilde{G}_2$ where
\begin{align*}
    &\widetilde{G}_1=C_\beta{\lambda_q}^{\beta+\alpha-2[\beta]-1}\nabla^{\perp}\!\left(\sum_{\xi}[(-\Delta)^{\beta},\bar{a}_\xi] e^{2\pi i \lambda_q\xi\cdot x}\right)e^{-(2\pi \lambda_q)^{2\beta}t},
    \\&\widetilde{G}_2=C_\beta{\lambda_q}^{\beta+\alpha-2[\beta]-1}\nabla^{\perp}\!\left(\sum_{\xi}(\partial_t\bar{a}_\xi) e^{2\pi i \lambda_q\xi\cdot x}\right)e^{-(2\pi \lambda_q)^{2\beta}t}.
\end{align*}
By \eqref{commuator}, we obtain
\begin{align}\label{g_1}
    \|\nabla^m \widetilde{G}_1\|_{L^{\infty}(\mathbb{T}^2)}\lesssim&
    \sum_{\xi}{\lambda_q}^{\beta+\alpha-2[\beta]-1}\| \nabla^{1+m}[(-\Delta)^{\beta},\bar{a}_\xi] e^{2\pi i \lambda_q\xi\cdot x}\|_{L^{\infty}(\mathbb{T}^2)} e^{-(2\pi \lambda_q)^{2\beta}t}
    \nonumber\\&\lesssim {\lambda_q}^{\beta+\alpha-2[\beta]+m}(\lambda_q^{2\beta-1}\lambda^4+\lambda^{2\beta+3})e^{-(2\pi \lambda_q)^{2\beta}t}
    \nonumber\\&\lesssim \lambda_q^{-\gamma}\cdot\lambda_q^{3\beta+\alpha-1+\gamma-2[\beta]+m}e^{-(2\pi \lambda_q)^{2\beta}t}
    \nonumber\\&\lesssim \eta t^{-\frac{3\beta+\alpha-1+\gamma-2[\beta]+m}{2\beta}},
\end{align}
for sufficiently large $q$. Using the Proposition \ref{low bound of U}, one derives that
\begin{align}\label{g_2}
    \|\nabla^m \widetilde{G}_2\|_{L^{\infty}(\mathbb{T}^2)}&\lesssim {\lambda_q}^{\beta+\alpha-2[\beta]+m}\lambda^{2\beta}e^{-(2\pi \lambda_q)^{2\beta}t}
    \nonumber\\&\lesssim \lambda_q^{1-2\beta-\gamma}\cdot t^{-\frac{3\beta+\alpha-1+\gamma-2[\beta]+m}{2\beta}}
    \nonumber\\&\lesssim \eta t^{-\frac{3\beta+\alpha-1+\gamma-2[\beta]+m}{2\beta}},
\end{align}
provided that $q$ is sufficiently large.
Combining \eqref{g_1} and \eqref{g_2}, for $m=0,1$, we conclude
\begin{align*}
    \|\nabla^m \widetilde{G}\|_{L^{\infty}(\mathbb{T}^2)}&\leq \eta t^{-\frac{3\beta+\alpha-1+\gamma-2[\beta]+m}{2\beta}}.
\end{align*}

\end{proof}

\section{Construction of the perturbation}\label{sec:4}
 
 The vector fields $\{v^{(i)}\}_{i\in \{1,2\}}$ constructed in Section \ref{sectprincipal} and $U$ from Section \ref{sectioninlfation} are almost solutions of (\ref{e:NSf}), whose residual terms are small in some certain subcritical norm. Consequently, we may introduce perturbations to these approximate solutions and thereby obtain exact solutions.

Firstly, we show the existence of $\{\omega^{(i)}\}_{i\in \{1,2\}}$ such that $\{v^{(i)}+\omega^{(i)}\}_{i\in \{1,2\}}$ solves (\ref{e:NSf}).
Recall that 
\begin{align*}
   \gamma\in(0,\min\{\frac{\beta-\alpha-1}{2},\frac{1-2\alpha}{2}\}],\quad \epsilon_2\in (0, \min\{\frac{\beta-\alpha-1}{2},\frac{1-2\alpha-\gamma}{2}\}],
\end{align*}
and
\begin{align*}
    \tau\in(\frac{\beta+\alpha-1+\gamma}{2\beta},\frac{\beta-\alpha-\epsilon_2}{2\beta}),\quad \kappa\in(0,\frac{2[\beta]-\beta-\alpha+1-\gamma}{2}].
\end{align*}
Moreover, $\epsilon_1$ is the parameter which depends on other parameters, which will be fixed later. Recall $I(a,b)$ is defined as 
\begin{gather*}
    I(a,b)=\int_0^1(1-s)^{-a}s^{-b}ds,\quad\forall a,b<1.
\end{gather*}
We define the space for the perturbation $\{\omega^{(i)}\}_{i\in \{1,2\}}$:
\begin{align*}
    X=\{\omega\in C((0,1]\times\mathbb{T}^2;\mathbb{R}^2);\|\omega\|_X<\infty\},
\end{align*}
with norm
\begin{align*}
    \| \omega\|_{X}=\sup_{t\in (0,1]}t^{\tau}\|\omega\|_{L^{\infty}}.
\end{align*}

\begin{prop}\label{Prop rho}
    Let $i\in\{0,1\}.$ There exist $\eta_0>0$ such that for all $\eta\in(0,\eta_0)$, there exist $\omega^{(i)}\in B_X(0,\eta)$ such that $u^{(i)}=v^{(i)}+\omega^{(i)}$ satisfy \eqref{e:NSf}. Moreover, $\omega^{(i)}\rightarrow0$ in $B^{-\beta}_{\infty,\infty}(\mathbb{T}^2)$ as $t\rightarrow 0$. 
\end{prop}

\begin{proof}
    Without loss of generality, we write $v$, $\omega$, $F$ and $\widetilde{F}$ to represent $v^{(i)}$,$\omega^{(i)}$,$F^{(i)}$ and $\widetilde{F}^{(i)}$, respectively, with $i\in \{1,2\}$.
    
    In order for $u=v+\omega$ to be a solution of \eqref{e:NSf} with initial data $v(0)$, it is necessary that the perturbation $\omega$ satisfy
    \begin{align*}
        \begin{cases}
            (\partial_{t}+(-\Delta)^{\beta})\omega+\mathbb{P}\operatorname{div}\, (v\otimes \omega+\omega\otimes v+\omega\otimes \omega)=\mathbb{P}\operatorname{div}\, F+\Delta^{[\beta]}\widetilde{F},\\
            \omega(0,\cdot)=0.
        \end{cases}
    \end{align*}
    Applying the Duhamel formula, we seek $\omega$ as a fixed point of the mapping
    \begin{align}\label{Tomega}
        T\omega(t)=\int^t_0 e^{(t-s)(-\Delta)^{\beta}}(\mathbb{P}\operatorname{div}(F-(\omega\otimes\omega+v\otimes\omega+\omega\otimes v))+\Delta^{[\beta]}\widetilde{F})ds,
    \end{align}
    It follows from the definition of $X$ and \eqref{2.7} that
    \begin{align*}
        \|T\omega\|_{X}&\lesssim
\sup_{0<t\le 1} t^{\tau}\Bigg(\int_0^t (t-s)^{-\frac{1}{2\beta}}
\Bigl(\|\omega(s)\|_{L^\infty}^2+\|\omega(s)\|_{L^\infty}\|v(s)\|_{L^\infty}+ \|F(s)\|_{L^\infty}\Bigr)ds \\&\quad + \int_0^t (t-s)^{-\frac{2[\beta]-\kappa}{2\beta}}\|\nabla^{\kappa}\widetilde{F}(s)\|_{L^\infty}ds\Bigg)
\\&\lesssim
\sup_{0<t\le 1} t^{\tau}\Bigg(\int_0^t (t-s)^{-\frac{1}{2\beta}}
\Bigl(s^{-2\tau}\|\omega\|_{X}^2+s^{-\tau}\|\omega\|_{X}\|v(s)\|_{L^\infty}+ \|F(s)\|_{L^\infty}\Bigr)ds \\&\quad + \int_0^t (t-s)^{-\frac{2[\beta]-\kappa}{2\beta}}\|\nabla^{\kappa}\widetilde{F}(s)\|_{L^\infty}ds\Bigg).
    \end{align*}
Applying \eqref{mvi2} and Proposition \ref{Prop 4.1}, we obtain
    \begin{align*}
         \|T\omega\|_{X}&\lesssim\sup_{0<t\le 1} t^{\tau}\Bigg(\int_0^t (t-s)^{-\frac{1}{2\beta}}
\Bigl(s^{-2\tau}\|\omega\|_{X}^2+\epsilon_1s^{-\tau}s^{-\frac{\beta+\alpha+\epsilon_2}{2\beta}}\|\omega\|_{X}+ \eta s^{-\frac{2\beta+2\alpha-1+\gamma}{2\beta}}\Bigr)ds \\&\quad + \int_0^t (t-s)^{-\frac{2[\beta]-\kappa}{2\beta}}\eta s^{-\frac{3\beta+\alpha-1+\gamma-2[\beta]+\kappa}{2\beta}}ds\Bigg)
    \end{align*}
Hence, by performing a change of variables, one gets that
\begin{align*}
   \|T\omega\|_{X}
&\lesssim I\!\left(\frac{1}{2\beta}, 2\tau\right)\|\omega\|_{X}^2\sup_{0<t\le 1}t^{1-\tau-\frac{1}{2\beta}}
+\epsilon_1 I\!\left(\frac{1}{2\beta}, \tau+\frac{\beta+\alpha+\epsilon_2}{2\beta}\right)\|\omega\|_{X}\sup_{0<t\le 1} t^{\frac{\beta-\alpha-\epsilon_2-1}{2\beta}}
\\&\quad + \eta I\!\left(\frac{1}{2\beta}, \frac{2\beta+2\alpha-1+\gamma}{2\beta}\right)\sup_{0<t\le 1} t^{\tau-\frac{2\alpha+\gamma}{2\beta}}
\\&\quad+ \eta I\!\left(\frac{2[\beta]-\kappa}{2\beta}, \frac{3\beta+\alpha-1+\gamma-2[\beta]+\kappa}{2\beta}\right)\sup_{0<t\le 1}t^{\tau-\frac{\beta+\alpha-1+\gamma}{2\beta}}.
\end{align*}
Observe that the Beta integrals appearing above are finite. Indeed, by the assumptions on $\alpha, 
\kappa, \gamma, \tau$ and $\epsilon_2$, we deduce that all the parameters of the Beta integral $I$ lie in $(-\infty,1)$:
\begin{align*}
   \frac{1}{2\beta},\,2\tau,\, \tau+\frac{\beta+\alpha+\epsilon_2}{2\beta},\, \frac{2\beta+2\alpha-1+\gamma}{2\beta}, \,\frac{2[\beta]-\kappa}{2\beta},\,\frac{3\beta+\alpha-1+\gamma-2[\beta]+\kappa}{2\beta} <1.
\end{align*}
Moreover, the exponents of $t$ in the supremum are nonnegative:
\begin{align*}
    {1-\tau-\frac{1}{2\beta}},\, {\frac{\beta-\alpha-\epsilon_2-1}{2\beta}}, \,{\tau-\frac{2\alpha+\gamma}{2\beta}}, \,{\tau-\frac{\beta+\alpha-1+\gamma}{2\beta}}\geq0,
\end{align*}
which implies that the supremum is finite.
Therefore, by choosing $\epsilon_1$ and $\eta$ small enough, there exists a constant $C>0$ such that 
\begin{align*}
    \|T\omega\|_{X}
&\le C\|\omega\|_{X}^2+\frac{1}{4}\|\omega\|_{X}+\frac{\eta}{2},
\end{align*}
By a similar argument, one infers
\begin{align*}
    \|T\omega_1-T\omega_2\|_{X}\le \|\omega_1-\omega_2\|_X(C\|\omega_1\|_{X}+C\|\omega_2\|_X+\frac{1}{4})
\end{align*}
Choosing $\eta_0=\frac{1}{8C}$, it follows that $T$ is a contraction on the ball $B_X(0,\eta)$ for all $\eta \in (0,\eta_0)$. We establish the existence of $\omega$ by the Banach fixed point theorem.

Finally, we verify that $\omega$ vanishes near the initial time in
$B^{-\beta}_{\infty,\infty}(\mathbb T^2)$. Applying \eqref{2.7}, \eqref{mvi}, Proposition \ref{Prop 4.1}, and the bound $\|\omega\|_{X} \le \eta $, we obtain
\begin{align*}
    \| \omega(t)\|_{B^{-\beta}_{\infty,\infty}(\mathbb{T}^2)}&\lesssim\int_0^t \|e^{(t-s)(-\Delta)^{\beta}}(F-(\omega\otimes\omega+v\otimes\omega+\omega\otimes v))\|_{B^{1-\beta}_{\infty,\infty}(\mathbb{T}^2)}ds
    \\
    &\quad+\int_0^t\|e^{(t-s)(-\Delta)^{\beta}}\Delta^{[\beta]}\nabla^{-\beta}\widetilde{F}\|_{L^{\infty}(\mathbb{T}^2)}ds
    \\&\lesssim \int_0^t\|(F-\omega\otimes \omega-v\otimes \omega-\omega\otimes v)(s)\|_{L^{\infty}(\mathbb{T}^2)}ds
    +\int_0^t (t-s)^{-\frac{2[\beta]-\beta}{2\beta}}\|\widetilde{F}(s)\|_{L^{\infty}(\mathbb{T}^2)}ds
    \\&\lesssim\int_0^t \left(s^{-\frac{2\beta+2\alpha-1+\gamma}{2\beta}}+ s^{-2\tau}+s^{-\tau-\frac{\beta+\alpha}{2\beta}}+(t-s)^{-\frac{2[\beta]-\beta}{2\beta}}s^{-\frac{3\beta+\alpha+\gamma-2[\beta]-1}{2\beta}}\right)ds
    \\&\lesssim t^{\frac{1-2\alpha-\gamma}{2\beta}}+ t^{1-2\tau}+t^{1-\tau-\frac{\beta+\alpha}{2\beta}}+I\left(\frac{2[\beta]-\beta}{2\beta},\frac{3\beta+\alpha+\gamma-2[\beta]-1}{2\beta}\right)t^{\frac{1-\alpha-\gamma}{2\beta}}
    \\&\underset{t\to0}{\longrightarrow}0,
\end{align*}
where we have used that all exponents of $t$ on the right-hand side are strictly positive.
\end{proof}

Similarly, we define the space for the perturbation $\rho$ of $U$ by
\begin{align*}
    X^{'}\triangleq\{\rho\in C^0((0,t_q]\times\mathbb{T}^2;\mathbb{R}^2);\|\rho\|_{X'}<\infty\},
\end{align*}
with norm
\begin{align*}
    \| \rho\|_{X^{'}}\triangleq\sup_{t\in (0,t_{q}]}t^{\tau}\|\rho\|_{L^{\infty}}<\infty,
\end{align*}
and the existence time
\begin{align*}
    t_q\triangleq\frac{1}{\lambda^{2\beta}_q}.
\end{align*}
 Unlike Proposition \ref{Prop rho}, we only need to establish the existence of the perturbation up to the norm inflation time $t_q$.

\begin{prop}
     There exist $\eta_0>0$ such that for all $\eta\in(0,\eta_0)$, there exist $\rho\in B_X(0,\eta)$ such that $u=U+\rho$ satisfy \eqref{e:NSf}. Moreover, for any $\epsilon>0$, one can choose $q$ sufficiently large such that
     \begin{align*}
         \|\rho\|_{B^{-\beta}_{\infty,\infty}(\mathbb{T}^2)} \le \epsilon.
     \end{align*}
     
\end{prop}
\begin{proof}
    In order for $u= U+ \rho$ to solve \eqref{e:NSf} with initial data $U(0)$, the perturbation $\rho$ must satisfy
    \begin{align*}
        \begin{cases}
            (\partial_{t}+(-\Delta)^{\beta})\rho+\mathbb{P}\operatorname{div}\, (U\otimes \rho+\rho\otimes U+\rho\otimes \rho)=\mathbb{P}\operatorname{div}\, G+\Delta^{[\beta]}\widetilde{G},\\
            \rho(0,\cdot)=0.
        \end{cases}
    \end{align*}
    Now by the Duhamel formula, we solve $\omega$ as a fixed point of the map
    \begin{align}\label{Tomega}
        T'\rho(t)=\int^t_0 e^{(t-s)(-\Delta)^{\beta}}(\mathbb{P}\operatorname{div}(G-(\rho\otimes\rho+U\otimes\rho+\rho\otimes U))+\Delta^{[\beta]}\widetilde{G})ds.   
    \end{align}
 Similarly to the proof of Proposition \ref{Prop rho}, one gets
\begin{align*}        
\|T'\rho\|_{X'}&\lesssim
\sup_{0<t\le t_q} t^{\tau}\int_0^t (t-s)^{-\frac{1}{2\beta}}
\Bigl(s^{-2\tau}\|\rho\|_{X}^2+ s^{-\tau}\|\rho\|_{X}\|U_2(s)\|_{L^\infty}+ \|G(s)\|_{L^\infty}\Bigr)ds \\&\quad +\sup_{0<t\le t_q} t^{\tau}\int_0^t (t-s)^{-\frac{1}{2\beta}}s^{-\tau}\|\rho\|_{X}\|U_1(s)\|_{L^\infty}ds+\sup_{0<t\le t_q} \int_0^t (t-s)^{-\frac{2[\beta]-\kappa}{2\beta}}\|\nabla^{\kappa}\widetilde{G}(s)\|_{L^\infty}ds.   
\end{align*}
Since $U_2$, $G$ and $\widetilde{G}$ satisfy the same estimates as
$v$, $F$ and $\widetilde{F}$, respectively, the proof is identical to
that of Proposition~\ref{Prop rho}, except for the treatment of the second term. In fact, using \eqref{estimateU1}, we have
\begin{align*}
    \sup_{0<t\le t_q} t^{\tau}\int_0^t (t-s)^{-\frac{1}{2\beta}}s^{-\tau}\|\rho\|_{X}\|U_1(s)\|_{L^\infty}ds\lesssim \lambda^{s+1}I(\frac{1}{2\beta},\tau)(t_q)^{1-\frac{1}{2\beta}}\|\rho\|_{X}\leq \frac{1}{8}\|\rho\|_{X}.
\end{align*}
provided  that $q$ is chosen to be sufficiently large.
Applying Propositions \ref{low bound of U} and \ref{Norm flation G}, and arguing as in the proof of Proposition \ref{Prop rho}, we arrive at
\begin{align*}
    \|T'\rho\|_{X'}
&\le C\|\rho\|_{X'}^2+\frac{1}{4}\|\rho\|_{X'}+\frac{\eta}{2},
\end{align*}
and
\begin{align*}
    \|T'\rho_1-T'\rho_2\|_{X'}\le \|\rho_1-\rho_2\|_{X'}(C\|\rho_1\|_{X'}+C\|\rho_2\|_{X'}+\frac{1}{4})
\end{align*}
where $\eta$ is chosen sufficiently small and $q$ chosen sufficiently large.
Choosing $\eta_0=\frac{1}{8C}$, it follows that $T'$ is a contraction map on the $B_X(0,\eta)$ for all $\eta \in (0,\eta_0)$, and we conclude by the Banach fix point theorem.

Finally, we derive the desired estimate for $\rho$ near the initial time. By \eqref{2.7}, \eqref{Norm flation G}, Proposition \ref{low bound of U}, the interpolation inequality, and the fact that $|\omega|_{X}\le \eta_1$, we obtain for $0<t<t_q$
\begin{align*}
    \| \rho(t)\|_{B^{-\beta}_{\infty,\infty}(\mathbb{T}^2)}
    &\lesssim\int_0^t \|e^{(t-s)(-\Delta)^{\beta}}(G-(\rho\otimes\rho+U_1\otimes\rho+U_2\otimes\rho+\rho\otimes U_1+\rho\otimes U_2)\|_{B^{1-\beta}_{\infty,\infty}(\mathbb{T}^2)}ds\\
    &\quad+\int_0^t\|e^{(t-s)(-\Delta)^{\beta}}\Delta^{[\beta]}\nabla^{-\beta}\widetilde{G})\|_{L^{\infty}(\mathbb{T}^2)}ds
    \\&\lesssim \int_0^t\|G(s)\|_{L^{\infty}(\mathbb{T}^2)}+\|\rho(s)\|_{L^{\infty}(\mathbb{T}^2)}^2+\|\rho(s)\|_{L^{\infty}(\mathbb{T}^2)}\|U_2(s)\|_{L^{\infty}(\mathbb{T}^2)}ds
    \\&\quad+\int_0^t\|\rho\|_{L^{\infty}(\mathbb{T}^2)}\| U_1\|_{L^{\infty}(\mathbb{T}^2)}ds+\int^t_0(t-s)^{-\frac{2[\beta]-\beta}{2\beta}}\|\widetilde{G}(s)\|_{L^{\infty}(\mathbb{T}^2)}ds.
\end{align*}
For the second term, choosing $q$ large enough, we deduce
\begin{align*}
    \int_0^t\|\rho(s)\|_{L^{\infty}(\mathbb{T}^2)}\| U_1(s)\|_{L^{\infty}(\mathbb{T}^2)}ds\lesssim \lambda^{s+1}\|\rho\|_{X'}\int^t_0s^{-\tau}ds\lesssim (t_q)^{1-\tau}\leq \frac{1}{3}\epsilon,\quad\forall t\in[0,t_q].
\end{align*}
While the estimates for the other terms are identical to those in Proposition \ref{Prop rho}, we omit the details.
\end{proof}

\section{Proof of the main results}\label{Proof of the Non-Uniqueness}
\begin{proof}[\textbf{Proof of Proposition \ref{mainprop1}}]
Recall from definitions, we can write
    \begin{align*}
    u^{(1)}\triangleq v^{(1)}+\omega^{(1)}=\sum_{k\geq0\,even}v_k+\sum_{k\geq0\,odd}\bar{v}_k+\omega^{(1)},
    \end{align*}
    and 
    \begin{align*}
        u^{(2)}\triangleq v^{(2)}+\omega^{(2)}=\sum_{k\geq0\,odd}v_k+\sum_{k\geq0\,even}\bar{v}_k+\omega^{(2)}.
    \end{align*}

In Section \ref{sec:4}, we have shown that for $i \in \{1,2\}$, there exist $\omega^{(i)} \in C((0,1]\times\mathbb{T}^2;\mathbb{R}^2)$ obeying suitable equations such that $\{u^{(i)}\}_{i\in\{1,2\}}$ obey \eqref{e:NSf}. Moreover, by standard regularity theory, $\{u^{(i)}\}_{i\in\{1,2\}}$ are smooth on $(0,1]\times \mathbb{T}^2$.

Next we argue that both solutions are continuous on $t=0$ with respect to the space $B^{-\beta-\alpha-\varepsilon'}_{\infty,\infty}(\mathbb{T}^2)$. As shown in Proposition \ref{Prop rho}, $\{\omega^{(i)} \}_{i\in\{0,1\}}$ tends to $0$ in $B^{-\beta}_{\infty,\infty}(\mathbb{T}^2)$ when $t$ tends to $0$. Thus, $\{\omega^{(i)} \}_{i\in\{0,1\}}$ belongs to $C([0,1];B^{-\beta-\alpha-\varepsilon'}_{\infty,\infty}(\mathbb{T}^2))$. It remains to show $\{v^{(i)} \}_{i\in\{0,1\}}\in C([0,1];B^{-\beta-\alpha-\varepsilon'}_{\infty,\infty}(\mathbb{T}^2))$. Since $\{v^{(i)} \}_{i\in\{0,1\}}\in C^{\infty}((0,1]\times\mathbb{T}^2)$, it suffice to prove that $\{v^{(i)} \}_{i\in\{0,1\}}$ are continuous at $t=0$ in $B^{-\beta-\alpha-\varepsilon'}_{\infty,\infty}(\mathbb{T}^2)$. Recalling the definition of $v^{(i)}$ in Definition \ref{defiofvk}, we have
\begin{align}
    \|v^{(i)}(t)-V^0\|_{B^{-\beta-\alpha-\varepsilon'}_{\infty,\infty}}\lesssim\sum_{k\in\mathbb{N}}\|v^0_k\|_{B^{-\beta-\alpha-\varepsilon'}_{\infty,\infty}}\left((e^{-(2\pi N_k)^{2\beta}t}-1)+(e^{-2\cdot(2\pi N_{k+1})^{2\beta}t}-1)\right).\label{estimaeintial}
\end{align}
Following the proof of Proposition \ref{propofinitial}, one can infers that
\begin{gather}
    \|v^0_k\|_{B^{-\beta-\alpha-\varepsilon'}_{\infty,\infty}}\lesssim N_k^{-\varepsilon'},\quad\quad\forall k\in\mathbb{N}. \label{-vareepsilon}
\end{gather}
Substituting \eqref{-vareepsilon} into \eqref{estimaeintial}, we arrive at
\begin{align*}
    \|v^{(i)}(t)-V^0\|_{B^{-\beta-\alpha-\varepsilon'}_{\infty,\infty}}\lesssim\sum_{k\leq M}
    N_k^{-\varepsilon'}\left((e^{-(2\pi N_k)^{2\beta}t}-1)+(e^{-2\cdot(2\pi N_{k+1})^{2\beta}t}-1)\right)+\sum_{k\geq M}
    N_k^{-\varepsilon'}\underset{t\to0}{\longrightarrow}0,
\end{align*}
where we first select large $M$ to make the second term small, then fix $M$ and take $t$ small to make the first term small as well.

Finally, we prove that these two solutions are distinct by considering a time scale when all but the lowest frequency mode should have dissipated away, for example $t_0=N_0^{-2\beta}$. By the triangle inequality, we get
\begin{align}
\|u^{(1)}(t_0)-u^{(2)}(t_0)\|_{L^\infty(\mathbb{T}^2)}
&\ge \|v_0\|_{L^\infty(\mathbb{T}^2)} - \sum_{k\ge 1}\|v_k(t_0)\|_{L^\infty(\mathbb{T}^2)}  - \sum_{k\ge 0}\|\bar{v}_k(t_0)\|_{L^\infty(\mathbb{T}^2)} \nonumber\\
 &\quad- \|\omega^{(1)}(t_0)\|_{L^\infty(\mathbb{T}^2)}  - \|\omega^{(2)}(t_0)\|_{L^\infty(\mathbb{T}^2)}.\label{all}
\end{align}
By \eqref{nablamvi}, one obtains 
\begin{align*}
    \sum_{k\ge 1}\|v_k(t_0)\|_{L^\infty(\mathbb{T}^2)}&\lesssim \sum_{k\ge 1}N_k^{\beta+\alpha}\exp{(-N_k^{2\beta}N_0^{-2\beta})}\lesssim N_1^{\beta+\alpha}\exp{(-N_1^{2\beta}N_0^{-2\beta})},
\end{align*}
and
\begin{align*}
    \sum_{k\ge 0}\|\bar{v}_k(t_0)\|_{L^{\infty}(\mathbb{T}^2)}\lesssim\sum_{k\ge 0}N_k^{\beta+\alpha}\exp{(-N_{k+1}^{2\beta}N_0^{-2\beta})}\lesssim N_0^{\beta+\alpha}\exp{(-N_1^{2\beta}N_0^{-2\beta})}.
\end{align*}
Therefore, we have
\begin{align}
    \sum_{k\ge 1}\|v_k(t_0)\|_{L^\infty(\mathbb{T}^2)}+ \sum_{k\ge 0}\|\bar{v}_k(t_0)\|_{L^{\infty}(\mathbb{T}^2)}\leq N_0^{\beta}.\label{y1}
\end{align}
 For the $\{\omega^{(i)}\}_{i\in\{1,2\}}$, we have by Proposition \ref{Prop rho} and definition of $\tau$, 
\begin{align}
    \|\omega^{(i)}(t_0)\|_{L^{\infty}(\mathbb{T}^2)}\lesssim \eta t_0^{-\tau}\lesssim N_0^{2\beta\tau}\lesssim N_0^{\beta-\alpha}\leq N_0^{\beta}.\label{y3}
\end{align}
Finally, we establish the exact value of $v_0(t_0,x)$ on $L^{\infty}(\mathbb{T}^2)$. By Definition \ref{Apropdefivk} and Definition \ref{defiofvk}, we readily obtain 
\begin{align}
    \|v_0(t_0,\cdot)\|_{L^{\infty}(\mathbb{T}^2)}= N_0^{\beta+\alpha}e^{-(2\pi)^{2\beta}}.\label{y5}
\end{align}
 Now together with \eqref{y1}-\eqref{y5}, we then have
\begin{align*}
    \|u^{(1)}(t_0)-u^{(2)}(t_0)\|_{L^\infty(\mathbb{T}^2)}\gtrsim N_0^{\beta+\alpha}e^{-(2\pi)^{2\beta}}-N_0^\beta.
\end{align*}
 By taking $A$ sufficiently large, we conclude that $u^{(1)}$ and $u^{(2)}$ are distinct.
\end{proof}

\begin{proof}[\textbf{Proof of Proposition \ref{mainprop2}}]
    By the definition of $U_1$ and $\rho$, we have
    \begin{align*}
        u(x,0)=U(x,0)+\rho(x,0)=U_1(x,0)+U_2(x,0)+\rho(x,0) =U_2(x,0).
    \end{align*}
     From the definition of $U_2$ in \eqref{U_1} and the proof of Proposition \ref{propofinitial}, it holds
    \begin{align*}
    \|u(\cdot,0)\|_{B^{-\beta-\alpha-\varepsilon}_{\infty,\infty}(\mathbb{T}^2)}&\le
        \|U_2(\cdot,0)\|_{B^{-\beta-\alpha-\varepsilon}_{\infty,\infty}(\mathbb{T}^2)}
         \\&\lesssim {\lambda_q}^{\beta+\alpha-2[\beta]-1} \left\|\Delta^{[\beta]}\nabla^{\perp}\!\left(\sum_{\xi}\bar{a}_{\xi}(\cdot ,0)e^{2\pi i \lambda_q\xi\cdot }\right)\right\|_{B^{-\beta-\alpha-\varepsilon}_{\infty,\infty}(\mathbb{T}^2)}
         \\&\lesssim {\lambda_q}^{-\varepsilon}.
    \end{align*}
    Hence, by taking $q$ sufficiently large, we conclude that
    \begin{align*}
        \|u(\cdot,0)\|_{B^{-\beta-\alpha-\varepsilon}_{\infty,\infty}(\mathbb{T}^2)}< \delta. 
    \end{align*}
    
We now prove that $U$ becomes large as claimed. Recall the norm inflation time $t_q=\frac{1}{\lambda_q^{2\beta}}$, $s>\beta+\alpha$ and $\lambda=\delta^{-1}$.
 We define $\widetilde{s}\triangleq \min\{s,2[\beta]+1\}$.   
     Using Definition \ref{propdefiU1} and Proposition \ref{Prop rho}, we arrive at
    \begin{align}\label{U_s}
        \|u(\cdot,t_q)\|_{B^{-s}_{\infty,\infty}(\mathbb{T}^2)}
&\ge \|U_1(\cdot,t_q)\|_{B^{-s}_{\infty,\infty}(\mathbb{T}^2)} - \|U_2(\cdot,t_q)\|_{B^{-s}_{\infty,\infty}(\mathbb{T}^2)}-\|\rho(\cdot,t_q)\|_{B^{-s}_{\infty,\infty}(\mathbb{T}^2)} \nonumber\\
&\gtrsim  \|U_1(\cdot,t_q)\|_{B^{-s}_{\infty,\infty}(\mathbb{T}^2)} - \|U_2(\cdot,t_q)\|_{B^{-\widetilde{s}}_{\infty,\infty}(\mathbb{T}^2)}-\|\rho(\cdot,t_q)\|_{B^{-\beta}_{\infty,\infty}(\mathbb{T}^2)}\nonumber
\\&\gtrsim \lambda e^{-(2\pi \lambda)^{2\beta}t_q}\big(1-e^{-2 (2\pi \lambda_q)^{2\beta}t_q}\big)
- \lambda_q^{\beta+\alpha-\widetilde{s}}e^{-\lambda_q^{2\beta}t_q}
-\epsilon
\nonumber\\
&\ge \delta^{-1},\nonumber
    \end{align}
which $q$ is chosen to be large enough.  
    
\end{proof}

\begin{proof}[\textbf{Proof of Theorem \ref{maintheo3}}]
     Let $\alpha<\min\{0,\beta-1\}$, $\beta>\frac{1}{2}$ and the dimension $d\geq2$. Suppose $u_0\in  \dot{B}^{-\beta-\alpha}_{\infty,\infty}(\mathbb{T}^2)$.
     For any $T>0$, we define
    \begin{align*}
        D_T\triangleq\{ u\in C((0,T)\times\mathbb{T}^2;\mathbb{R}^2); \nabla \cdot u=0 \,\,\text{and} \,\, \|u\|_{D_T}<\infty \}, 
    \end{align*}
    with norm
    \begin{align*}
        \|u\|_{D_T}\triangleq\sup_{t\in (0,T]}t^{\frac{\beta+\alpha}{2\beta}}\|u(t)\|_{L^{\infty}(\mathbb{T}^d)}.
    \end{align*}
     
     The existence time $T(u_0)$ is defined by
\begin{align*}
    T(u_0)\triangleq\left(8C_1\|u_0\|_{\dot{B}_{\infty,\infty}^{-(\beta+\alpha)}(\mathbb{T}^d)}I\left(\frac{1}{2\beta},\frac{\beta+\alpha}{\beta}\right)\right)^{\frac{2\beta}{\alpha+1-\beta}},
\end{align*}
where $C_1$ is a universal constant appearing in the subsequent estimates.
 We will use the Banach fixed point theorem to prove the existence of solutions on $D_{T(u_0)}$.
The mild solution to \eqref{e:NSf} is characterized as the fixed point of the operator
    \begin{align*}
        (Au)(t,x) = e^{-t(-\Delta)^\beta} u_0(x) - \int_0^t e^{-(t-s)(-\Delta)^\beta} \mathbb{P}\operatorname{div}(u\otimes u)(s,x)\,ds
    \end{align*}

   We first estimate $\| Au(t)\|_{L^{\infty}(\mathbb{T}^d)}$ for $t \in (0,T(u_0)]$,
    \[
\|Au(t)\|_{L^\infty(\mathbb{T}^d)}
\le \|e^{-t(-\Delta)^\beta} u_0\|_{L^\infty(\mathbb{T}^d)}
+ \int_0^t (t-s)^{-\frac{1}{2\beta}} \|u(s)\|_{L^\infty(\mathbb{T}^d)}^2 \, ds.
\]
Then, using an equivalent characterization of Besov spaces and a change of variables, we obtain
\[
\begin{aligned}
\|Au\|_{D_{T(u_0)}}
&\le \sup_{t\in (0,T(u_0)]} t^{\frac{\beta+\alpha}{2\beta}} \|e^{-t(-\Delta)^\beta} u_0\|_{L^\infty(\mathbb{T}^d)}
+ \sup_{t\in(0,T(u_0)]} t^{\frac{\beta+\alpha}{2\beta}} \int_0^t (t-s)^{-\frac{1}{2\beta}} \|u(s)\|_{L^\infty(\mathbb{T}^d)}^2  \, ds \\
&\le C_1 \|u_0\|_{\dot{B}_{\infty,\infty}^{-(\beta+\alpha)}(\mathbb{T}^d)}
+ \|u\|_{D_{T(u_0)}}^2 \sup_{t\in(0,T(u_0)]} t^{\frac{\beta+\alpha}{2\beta}} \int_0^t (t-s)^{-\frac{1}{2\beta}} s^{-\frac{\beta+\alpha}{\beta}}  ds
\\&\le C_1 \|u_0\|_{\dot{B}_{\infty,\infty}^{-(\beta+\alpha)}(\mathbb{T}^d)}
+ I\left(\frac{1}{2\beta},\frac{\beta+\alpha}{\beta}\right)\|u\|_{D_{T(u_0)}}^2 \sup_{t\in(0,T(u_0)]} t^{\frac{\beta-\alpha-1}{2\beta}}.
\end{aligned}
\]

Using the facts that \(\alpha<\min\{0,\beta-1\}\) and \(\beta>\frac12\), we see that both coefficients in the Beta integral $I$ are less than 1, while the exponent of t appearing inside the supremum is strictly positive. Substitute the definition of \(T(u_0)\) into the supremum, we obtain
\begin{align*}
    \|Au\|_{D_{T(u_0)}}\le C_1 \|u_0\|_{\dot{B}_{\infty,\infty}^{-(\beta+\alpha)}(\mathbb{T}^d)}+(8C_1 \|u_0\|_{\dot{B}_{\infty,\infty}^{-(\beta+\alpha)}(\mathbb{T}^d)})^{-1}\|u\|_{D_{T(u_0)}}^2 
\end{align*}
Moreover, by a analogous way, it follows
\begin{align*}
\|Au_1 - Au_2\|_{D_{T(u_0)}} &\le (8C_1 \|u_0\|_{\dot{B}_{\infty,\infty}^{-(\beta+\alpha)}(\mathbb{T}^d)})^{-1}(\|u_1\|_{D_{T(u_0)}}+\|u_2\|_{D_{T(u_0)}})\|u_1-u_2\|_{D_{T(u_0)}}
.
\end{align*}
Therefore, $A$ is a contraction mapping in $B_{D_{T(u_0)}}(0,r_0)$ with $r_0=2 C_1 \|u_0\|_{\dot{B}_{\infty,\infty}^{-(\beta+\alpha)}(\mathbb{T}^d)}$. By the Banach fixed point theorem, we obtain the existence of a solution with initial data \(u_0\).
\end{proof}

\begin{proof}[\textbf{Proof of Theorem \ref{maintheo4}}]
      Let $\alpha<0$ with $\alpha\leq\beta-1$, $\beta>\frac{1}{2}$, the dimension $d\geq2$ and $T>0$. Suppose $u_0\in  \dot{B}^{-\beta-\alpha}_{\infty,\infty}(\mathbb{T}^2)$ satisfying
\begin{align*}
    \|u_0\|_{\dot{B}_{\infty,\infty}^{-\beta-\alpha}(\mathbb{T}^d)}\leq\frac{1}{8C_1C_2}.
\end{align*}
where $C_1$ is the universal constant appearing in the proof of Theorem \ref{maintheo3}, and $C_2$ is a constant to be specified later.
        As in Theorem \ref{maintheo3}, it yields the following estimate
\[
\begin{aligned}
\|Au\|_{D_T}
&\le C_1 \|u_0\|_{\dot{B}_{\infty,\infty}^{-(\beta+\alpha)}(\mathbb{T}^d)}
+ I\left(\frac{1}{2\beta},\frac{\beta+\alpha}{\beta}\right)\|u\|_{D_T}^2 \sup_{t\in(0,T]} t^{\frac{\beta-\alpha-1}{2\beta}} 
\\&\le C_1\|u_0\|_{\dot{B}_{\infty,\infty}^{-\beta-\alpha}(\mathbb{T}^d)} + C_2 \|u\|_{D_T}^2.
\end{aligned}
\]
where we have used the assumptions on $\alpha$ and $\beta$, and $C_2$ is a constant depending on $\alpha$, $\beta$ and $T$. Similarly, it follows that, for any $u_1, u_2 \in {D_T}$.
\[
\|Au_1 - Au_2\|_{D_T} \le C_2 (\|u_1\|_{D_T}+\|u_2\|_{D_T} )\|u_1 - u_2\|_{D_T}.
\]
Since $\|u_0\|_{\dot{B}_{\infty,\infty}^{-\beta-\alpha}(\mathbb{T}^d)}\leq\frac{1}{8C_1C_2}$, $A$ is a contraction mapping in $B_{D_T}(0,r_0)$ with $r_0=2C_1\|u_0\|_{\dot{B}_{\infty,\infty}^{-\beta-\alpha}(\mathbb{T}^d)}$. By the Banach fixed point theorem, we obtain the existence of a solution with initial data \(u_0\).

Moreover, for the case $\beta\in(\frac{1}{2},1)$, $\alpha=\beta-1$, the existence time $T$ can be taken to be infinite and the proof is analogous. Indeed, Since $\alpha=\beta-1<0$, the required estimate follows from
\begin{align*}
    I\left(\frac{1}{2\beta},\frac{\beta+\alpha}{\beta}\right)\sup_{t\in(0,\infty]} t^{\frac{\beta-\alpha-1}{2\beta}}\lesssim1.
\end{align*}

\end{proof}

\appendix

\section{Tools from convex integration }
In this section, we recall some useful tools which is developed in the convex integration method. 

Let $\mathbb{S}^1$ be the unit circle and $\mathbb{Q}$ the set of rational numbers. We denote $B_{\frac{1}{7}}(\mathrm{Id})$ is the ball of radius $\frac{1}{7}$ centered at $\mathrm{Id}$ in the space of $2 \times 2$ symmetric matrices. The following geometric lemma, originally due to \cite{nashgeomtric} and reformulated in \cite{2dgeomtric}, will play a key role.

\begin{lemm}\label{geometric}
    There exists a finite set $\Lambda \subset \mathbb{S}^{1} \cap \mathbb\,{\mathbb{Q}}^2$ such that $\Lambda=-\Lambda$ and each $\xi\in\Lambda$ is equipped with an orthonormal basis $(\xi,\bar{\xi})$ satisfying $\overline{(-\xi)}=-\bar{\xi}$. Moreover, there exist smooth positive functions $a_{\xi}: B_{\frac{1}{7}}(\mathrm{Id}) \to \mathbb{R}$ satisfying $a_{\xi}=a_{-\xi}$, such that we have the following identity:
\begin{gather}
    S = \sum_{\xi \in \Lambda} a_{\xi}^2(S) \bar{\xi} \otimes \bar{\xi} ,\quad\forall S\in B_{\frac{1}{7}}(\mathrm{Id}).
\end{gather}
\end{lemm}

We recall the antidivergence operators $\mathcal{R}$ that are introduced in \cite{D1}. 
\begin{lemm}\label{def of antidiv}
    There exists a linear operator $\mathcal{R}:C^\infty(\mathbb{T}^2;\mathbb{R}^2)\rightarrow C^\infty(\mathbb{T}^2;S^{2\times 2})$ satisfying
    \begin{align*}
        \operatorname{div}\,\mathcal{R}f&=f-\fint_{\mathbb{T}^2}f,\quad\forall f\in C^\infty(\mathbb{T}^2;\mathbb{R}^2),
    \end{align*}
\end{lemm}

Variants of the following stationary phase lemma can be found in the works \cite{2,D1}.
\begin{lemm}
    Let $\mu>0$, for any smooth function $f$ and $m>0$, we have
    \begin{align}
        \|\mathcal{R}(f(x)e^{2\pi i\lambda kx})\|_{C^{\mu}}\lesssim_{\mu,m}\lambda^{-1+\mu}\|f\|_{L^\infty}+ \lambda^{-m+\mu}\|\nabla^mf\|_{L^\infty}+\lambda^{-m}\|\nabla^mf\|_{C^{\mu}},\label{eR}
    \end{align}
\end{lemm}

The following stationary flow lemma is originally due to \cite{2dlemmaarxiv}. 
\begin{lemm}\label{2dlemma}
     Given $\Lambda\subseteq\mathbb{S}^1\cap\mathbb{Q}^2$ such that $-\Lambda=\Lambda$. Then for any $b_k\in\mathbb{R}$ with $b_k=b_{-k}$, the vector field
\[
W(x)=\sum_{\xi\in\Lambda}b_\xi \mathrm{i}\bar{\xi}e^{2\pi\mathrm{i}\xi\cdot x},\ \xi\perp\bar{\xi},
\]
is real-valued, divergence-free and satisfies
\[
\operatorname{div}_x(W\otimes W)=\frac{1}{2}\nabla_x\left(|W|^2-|\sum_{\xi\in\Lambda}b_\xi e^{2\pi\mathrm{i}\xi\cdot x}|^2\right).
\]
\end{lemm}

\section{Fractional Leibniz commutator estimate}

\begin{lemm}
    Let $P,Q,s>0$ and $\xi\in\mathbb{T}^2\setminus\{0\}$. For any smooth functions $f$ satisfying 
    \begin{gather}
        \|\nabla^nf\|_{_{L^\infty(\mathbb{T}^2)}}\lesssim P^n,\quad\quad\forall n\in\mathbb{N},\label{nf}
    \end{gather}
     we have
    \begin{gather}
        \|\nabla^m[(-\Delta)^{s},f] e^{2\pi i Q\xi\cdot x}\|_{L^\infty(\mathbb{T}^2)}\lesssim (Q+P)^m(Q^{2s-1}P^4+P^{2s+3}),\label{commuator}
    \end{gather}
    for any $m\in\mathbb{N}$.
\end{lemm}

\begin{proof}
    Taking the fourier series of $f$, 
    \begin{gather*}
        f(x)=\sum_{l\in\mathbb{Z}^2}\hat{f}(l)e^{2\pi i l\cdot x}.
    \end{gather*}
     By \eqref{nf} and Parseval's identity, we obtain
      \begin{gather*}
\|\nabla^mf\|_{L^2(\mathbb{T}^2)}=\left(\sum_{l\in\mathbb{Z}^2}|l|^{2m}|\hat{f}(l)|^2\right)^{\frac{1}{2}}\lesssim \|\nabla^m f\|_{L^\infty(\mathbb{T}^2)}\lesssim P^m,\quad\quad\forall m\in\mathbb{N},
      \end{gather*}
      which immediately implies
      \begin{gather}
          |l|^{m}|\hat{f}(l)|\lesssim P^m,\quad\quad\forall (l,m)\in \mathbb{Z}^2\times \mathbb{N}.
      \end{gather}
    we compute
    \begin{align*}
        [(-\Delta)^{s},f(x)]e^{2\pi i Q\xi\cdot x}&=(-\Delta)^{s}(f(x)e^{2\pi i Q\xi\cdot x})-f(x)(-\Delta)^{s}e^{2\pi i Q\xi\cdot x}\\
        &=\sum_{l\in\mathbb{Z}^2}(2\pi)^{2s}\left(|l+Q\xi|^{2s}-|Q\xi|^{2s}\right)\hat{f}(l)e^{2\pi i(l+Q\xi)\cdot x}\\
        &=\left(\sum_{0<|l|\leq |Q\xi|}+\sum_{|l|>|Q\xi|}\right)(2\pi)^{2s}\left(|l+Q\xi|^{2s}-|Q\xi|^{2s}\right)\hat{f}(l)e^{2\pi i(l+Q\xi)\cdot x}\\
        &\triangleq L_1+L_2.
    \end{align*}
    For $L_1$, we estimate
    \begin{align}
        \|L_1\|_{L^\infty(\mathbb{T}^2)}&\lesssim\sum_{0<|l|\leq |Q\xi|}Q^{2s}\left(|\frac{l}{Q}+\xi|^{2s}-|\xi|^{2s}\right)|\hat{f}(l)|\nonumber\\
        &\lesssim \sum_{0<|l|\leq |Q\xi|}Q^{2s}\cdot |\frac{l}{Q}||\hat{f}(l)|\nonumber\\
        &\lesssim  Q^{2s-1} \max_l |l|^4|\hat{f}(l)| \sum_{l\neq0}\frac{1}{|l|^3}\nonumber\\
        &\lesssim  Q^{2s-1} P^4.\label{L1}
    \end{align}
    For $L_2$, we get
    \begin{align}
        \|L_2\|_{L^\infty(\mathbb{T}^2)}\lesssim\sum_{|l|>|Q\xi|}|l|^{2s}|\hat{f}(l)|
        \lesssim \max_l |l|^{2s+3}|\hat{f}(l)|\sum_{l\neq0}\frac{1}{|l|^3}
        \lesssim P^{2s+3}.\label{L2}
    \end{align}
    Combining \eqref{L1} and \eqref{L2} together, we obtain \eqref{commuator} for the case $m=0$. For $m\geq1$, consider the case $m=1$ as an example, we have
    \begin{align*}
        \nabla[(-\Delta)^{s},f]e^{2\pi i Q\xi\cdot x}=[(-\Delta)^{s},\nabla f]e^{2\pi i Q\xi\cdot x}+[(-\Delta)^{s},f]\nabla e^{2\pi i Q\xi\cdot x}.
    \end{align*}
  The desired estimates follow by a similar argument as above, so we omit the details.
\end{proof}

\section{Heat estimates for the linear heat semigroup}

Now recall the heat estimates. 
\begin{lemm}
    For all $m \ge 0$ and $f \in L^\infty(\mathbb{T}^2)$,
we have
\begin{align}\label{2.7}
    \|e^{-t(-\Delta)^\beta} \nabla^m f\|_{L^\infty(\mathbb{T}^2)} \lesssim_{m} t^{-\frac{m}{2\beta}} \|f\|_{L^\infty(\mathbb{T}^2)},\quad\forall t>0.
\end{align}
Moreover, the above estimate holds as well if $f$ on the left-hand side is replaced by $\mathbb{P}f$ with $m>0$.
\end{lemm}

\smallskip
\noindent\textbf{Acknowledgments} This work was partially supported by the National Natural Science Foundation of China (No.12571261). 

%(No.11671407 and No.11701586), the Macao Science and Technology Development Fund (No. 098/2013/A3), and Guangdong Province of China Special Support Program (No. 8-2015),
%and the key project of the Natural Science Foundation of Guangdong province (No. 2016A030311004).

%The authors thank the referee for valuable comments and suggestions.

\phantomsection
\addcontentsline{toc}{section}{\refname}
%Ìí¼Ó²Î¿¼ÎÄÏ×µ½ÊéÇ©£¬ºê°ü hyperref
%\bibliographystyle{abbrv} %plain ,%alpha, %abbrv
%\bibliography{Reference}

\printbibliography

\end{document}